\documentclass[11pt,leqno,twoside]{amsart}
\usepackage{amssymb,amsmath,amsthm,soul,color}
\usepackage{t1enc}
\usepackage[cp1250]{inputenc}
\usepackage{a4,indentfirst,latexsym}
\usepackage{graphics}
\usepackage{mathrsfs}
\usepackage{cite,enumitem,graphicx}
\usepackage[colorlinks=true,urlcolor=blue,
citecolor=red,linkcolor=blue,linktocpage,pdfpagelabels,
bookmarksnumbered,bookmarksopen]{hyperref}
\usepackage[english]{babel}
\usepackage[left=2.61cm,right=2.61cm,top=2.72cm,bottom=2.72cm]{geometry}
\usepackage[metapost]{mfpic}
\usepackage[hyperpageref]{backref}
\usepackage[colorinlistoftodos]{todonotes}
\usepackage[normalem]{ulem}

\makeatletter
\providecommand\@dotsep{5}
\def\listtodoname{List of Todos}
\def\listoftodos{\@starttoc{tdo}\listtodoname}
\makeatother

\numberwithin{equation}{section}
\newtheorem{Th}{Theorem}[section]

\newtheorem{Lem}[Th]{Lemma}
\newtheorem{lemma}[Th]{Lemma}

\newtheorem{Rem}[Th]{Remark}

\def\supp{\mathrm{supp}}

\def\R{\mathbb{R}}

\def\J{\mathcal{J}}

\def\RT{\mathbb{R}^3}
\def\n{\nabla}

\title[Nonlinear Schr\"odinger-Bopp-Podolsky system]{Nonlinear Schr\"odinger equation\\in the Bopp-Podolsky electrodynamics:\\solutions in the electrostatic case}
\author[P. d'Avenia]{Pietro d'Avenia}
\author[G. Siciliano]{Gaetano Siciliano}
\address[P. d'Avenia]{\newline\indent
	Dipartimento di Meccanica, Matematica e Management
	\newline\indent 
	Politecnico di Bari
	\newline\indent
	Via Orabona 4,  70125  Bari, Italy}
\email{\href{mailto:pietro.davenia@poliba.it}{pietro.davenia@poliba.it}}

\address[G. Siciliano]{\newline\indent
	Departamento de Matem\'atica - Instituto de Matem\'atica e Estat\'istica
	\newline\indent 
	Universidade de S\~ao Paulo
	\newline\indent
	Rua do Mat\~ao 1010,  05508-090  S\~ao Paulo, Brazil}
\email{\href{mailto:sicilian@ime.usp.br}{sicilian@ime.usp.br}}

\thanks{P. d'Avenia was partially supported by a grant of the group GNAMPA of INdAM and by FRA2016 of Politecnico di Bari. 
G. Siciliano was partially supported by Fapesp, CNPq and Capes, Brazil. This work was partially carried out during a stay of P. d'Avenia at the Universidade de S\~{a}o Paulo. He would like to express
his deep gratitude to the Instituto de Matem\'atica e Estat\'istica, Departamento de Matem\'atica, for the warm hospitality.
}

\begin{document}
	\begin{abstract} 
		We study the following nonlinear Schr\"odinger-Bopp-Podolsky system
		\[
\begin{cases}
-\Delta u + \omega u + q^{2}\phi u = |u|^{p-2}u\\
-\Delta \phi + a^2 \Delta^2 \phi = 4\pi  u^2
\end{cases}
\hbox{ in }\mathbb{R}^3
\]
with $a,\omega>0$.
We prove existence and nonexistence results depending on the parameters
$q,p$. Moreover we also show that, in the radial case, the solutions we find tend to solutions
of the classical Schr\"odinger-Poisson system as $a\to0$.
	\end{abstract}

\maketitle
\begin{center}
\begin{minipage}{12cm}
\tableofcontents
\end{minipage}
\end{center}

\section{Introduction}
In this paper we consider the system
\begin{equation}\label{eq:ScBP}
	\begin{cases}
		-\Delta u + \omega u + q^{2}\phi u = |u|^{p-2}u\\
		-\Delta \phi + a^2 \Delta^2 \phi = 4\pi  u^2
	\end{cases}
	\hbox{ in }\mathbb{R}^3
\end{equation}
where $u,\phi:\mathbb{R}^3 \to \R$, $\omega, a>0$, $q\neq 0$.

Such a system appears when we couple a Schr\"odinger field $\psi=\psi(t,x)$ with its electromagnetic field in the Bopp-Podolsky electromagnetic theory, and, in particular, in the electrostatic case for standing waves $\psi(t,x)=e^{i\omega t} u(x)$, see Section \ref{deduction} for more details.

The Bopp-Podolsky theory, developed by Bopp \cite{B40}, and independently by Podolsky \cite{Pob42}, is a second order gauge theory for the electromagnetic field. As the Mie theory \cite{Mie13} and its generalizations given by Born and Infeld \cite{Bnat,B,BInat,BI}, it was introduced to solve the so called {\em infinity problem} that appears in the classical Maxwell theory.
In fact, by the well known Gauss law (or Poisson equation), the electrostatic potential $\phi$ for a given charge distribution whose density is $\rho$ satisfies the equation
\begin{equation}
	\label{Gauss}
	-\Delta\phi=\rho
	\qquad
	\hbox{in }\R^3.
\end{equation}
If $\rho=4\pi\delta_{x_0}$, with $x_0\in\R^3$, the fundamental solution of \eqref{Gauss} is $\mathcal{G}(x-x_0)$, where
\[
\mathcal{G}(x)=\frac{1}{|x|},
\]
and the electrostatic energy is
\[
\mathcal{E}_{\rm M}(\mathcal{G})
=\frac{1}{2}\int_{\R^3} |\nabla \mathcal{G}|^2 =+\infty.
\]
Thus, equation \eqref{Gauss} is {\em replaced} by
\[
-\operatorname{div}\left(\frac{\nabla\phi}{\sqrt{1-|\nabla\phi|^2}}\right)=\rho
\qquad
\hbox{in }\R^3
\]
in the Born-Infeld theory and by
\[
-\Delta \phi + a^2 \Delta^2 \phi = \rho
\qquad
\hbox{in }\R^3
\]
in the Bopp-Podolsky one. In both cases, if $\rho=4\pi\delta_{x_0}$, we are able to write explicitly the solutions of the respective equations and to see that their energy is finite. In particular, when we consider the operator $-\Delta + a^2\Delta^2$, we have that $\mathcal{K}(x-x_0)$, with
\[
\mathcal{K}(x):=\frac{1-e^{-|x|/a}}{|x|},
\]
is the fundamental solution of the equation
\[
-\Delta \phi + a^2\Delta^2\phi= 4\pi\delta_{x_0},
\]
it has no singularity in $x_0$ since it satisfies
\[
\lim_{x\to x_0} \mathcal{K}(x-x_0)=\frac{1}{a},
\]
and its energy is
\[
\mathcal{E}_{\rm BP}(\mathcal{K})
=\frac{1}{2}\int_{\R^3} |\nabla \mathcal{K}|^2 
+\frac{a^2}{2} \int_{\R^3} |\Delta \mathcal{K}|^2 
<+\infty
\]
(see Section \ref{secop} for more details).

Moreover the Bopp-Podolsky theory may be interpreted as an effective theory for short distances (see \cite{F96}) and for large distances it is experimentally indistinguishable from the Maxwell one. Thus the Bopp-Podolsky parameter $a>0$, which has dimension of the inverse of mass, can be interpreted as  a cut-off distance or can be linked to an effective radius for the electron.
For more physical details we refer the reader to the recent papers \cite{BPVZ,BPO,BPS2014,BPS2017,CDMPP,CDMNP} and to references therein. 

Finally we point out that the operator $-\Delta+\Delta^2$ appears also in other different interesting mathematical and physical situations (see \cite{BCMN,FIP} and their references).

Before stating our results, few preliminaries are in order.
We introduce here 
the space $\mathcal D$ as the completion of  $C^{\infty}_{c}(\mathbb R^{3})$ with respect to the norm~$\sqrt{\|\nabla \phi\|_{2}^2+a^2\|\Delta\phi\|_2^2}$.
We refer to Section \ref{sec:preliminaries} for more properties on this space.\\
Then, fixed $a>0$ and $q\neq 0$, we say that a pair $( u, \phi)\in H^{1}(\mathbb R^{3})\times \mathcal D$
is a solution of \eqref{eq:ScBP} if
\begin{align*}
\begin{array}{ll}
\displaystyle
\int_{\R^3} \nabla  u \nabla v +\omega\int_{\R^3}  u v 
+q^{2}\int_{\R^3} \phi  u v
=\int_{\R^3} | u|^{p-2}  u v 
&
\hbox{for all } v\in H^{1}(\mathbb R^{3}),\smallskip
\\
\displaystyle
\int_{\R^3} \nabla \phi \nabla \xi +a^{2}\int_{\R^3} \Delta \phi \Delta \xi = 4\pi \int_{\R^3} \phi u^{2}
&
\hbox{for all }  \xi \in \mathcal D.
\end{array}
\end{align*}
By standard arguments the solutions we find are easily seen to be positive.
Moreover we say that a solution $( u,\phi)$ is nontrivial whenever $ u \not\equiv 0$.

Then our results can be stated as follows.

\begin{Th}[Existence for small charges]\label{th:existence}
If $p\in(2,6)$ then there exists $q_{*}>0$ such that, for all
$q\in(-q_{*}, q_{*})\setminus\{0\}$, problem \eqref{eq:ScBP} admits a nontrivial solution.
\end{Th}

The above existence result can be extended to any value of $q$ provided that
a restriction on $p$ is made.

\begin{Th}[Existence for arbitrary charges]\label{th:existence2}
If $p\in(3,6)$ then,  for all $q\neq0$,
problem \eqref{eq:ScBP} admits a nontrivial solution.
\end{Th}


%

\medskip

To prove our existence results we use Variational Methods. Indeed
the solutions can be found as critical points of a smooth functional $\mathcal J_{q}$
defined on  $H^{1}(\mathbb R^{3})$. However 
we need to face with the following difficulties.
The geometry of the functional  strongly depends on the values of the parameters $p$ and $q$,
that may allow or prevents the existence of critical points. 
Moreover, a fundamental tool in Critical Point Theory is the following compactness condition: we say that the functional $\J_q$ satisfies the {\sl Palais-Smale} ((PS) for short) condition if any (PS) sequence $\{u_n\}\subset H^1(\R^3)$, that is a sequence satisfying
	\[
	\{\J_q(u_n)\} \hbox{ bounded and }  \J'_q(u_n)\to 0 \hbox{ in } H^{-1}(\R^3)\hbox{ as }  n\to +\infty,
	\]
	admits a convergent subsequence.\\
In particular, to find a bounded (PS) sequence in the case $p\in(2,3]$, a suitable truncation, introduced in \cite{JL} and already used successfully in recent papers as \cite{ADP,FS,Kik}, is performed. Once we have a bounded (PS) sequence, the invariance by translations of the problem clearly leads to a second difficulty:  the  lack of compactness. To overcome this problem  a useful Splitting Lemma (see Lemma \ref{lemSplitting})
is given.\\
Another difficulty which appears is due to the fact that the kernel
$\mathcal K$
is not homogeneous. This
make difficult the use of rescaling of type $t\mapsto u(t^{\alpha}\cdot)$ and
hence arguments as in \cite{RuizJFA} cannot be used.
However we can take some advantage from the fact that $\mathcal{K}\leq\mathcal{G}$.

%
%

Let us observe that, due to the invariance of $\J_{q}$ under the group induced by the action of rotations on $H^1(\R^3)$, we can restrict ourselves
to $H_r^1(\R^3)$, the subspace of radial functions, which is a natural constraint: if $u\in H_r^1(\R^3)$ is a critical point of $\J_q|_{H_r^1(\R^3)}$, then it is a critical point for the functional on the entire $H^1(\R^3)$. 
Then
the same results as in Theorem \ref{th:existence} and Theorem \ref{th:existence2} hold in the radial setting
(with even a simpler argument in order to recover compactness).
Actually in this case we can say even more:
the solutions found tend to solutions of the Schr\"odinger-Poisson system
\begin{equation}\label{eq:SP}
\begin{cases}
-\Delta u + \omega u +  q^{2}\phi u = |u|^{p-2}u\\
-\Delta \phi   = 4\pi  u^2
\end{cases}
\hbox{ in }\mathbb{R}^3,
\end{equation}
obtained formally by \eqref{eq:ScBP} by setting $a=0$. Indeed we have the following
\begin{Th}\label{th:ato0}
For $ q\neq 0$ fixed according to the restriction in the  Theorems \ref{th:existence} and \ref{th:existence2}, let $(\mathfrak u^{a}, \phi^{a})\in H^{1}_{r}(\mathbb R^{3})\times \mathcal D_{r}$ be solutions of \eqref{eq:ScBP}. Then
\begin{equation*}
\mathfrak u^{a}\to \mathfrak u^{0}\text{ in } H^{1}_{r}(\mathbb R^{3}) \quad
\text{ and }\quad \phi^{a}\to \phi^{0} \text{ in } D_{r}^{1,2}(\mathbb R^{3})
\qquad \text{as } a\to0,
\end{equation*}
where $(\mathfrak u^{0},\phi^{0})\in H^{1}_{r}(\mathbb R^{3})\times D^{1,2}_{r}(\mathbb R^{3})$
is a solution of \eqref{eq:SP}.
\end{Th}

Let us recall that $D^{1,2}(\mathbb R^{3})$ is the usual Sobolev space
defined as the completion of  $C^{\infty}_{c}(\mathbb R^{3})$ with respect to the norm $\|\nabla \phi\|_{2}$ and $D^{1,2}_{r}(\mathbb R^{3}), \mathcal D_{r}$ are the respective subspaces of radial functions.

We point out that there is a wide literature on the coupling of matter with its own electromagnetic field for many different situations. See e.g. \cite{ADP,DPS} for the Maxwell theory, \cite{APS,BDP,DP,FS} for the Born-Infeld one, \cite{CDPS} for the Chern-Simon one, and references therein.
However, to the best of our knowledge, this is the first time that a system like \eqref{eq:ScBP},
which involves the Bopp-Podolsky electromagnetic theory,
appears in the mathematical literature.

The paper is organised as follows.\\
In Section  \ref{deduction} the Schr\"odinger-Bopp-Podolsky system is obtained in the
framework of the Abelian Gauge Theories.  Then the antsaz of stationary solutions in the electrostatic case
is done and \eqref{eq:ScBP} is deduced.\\
In Section \ref{sec:preliminaries} we give general preliminaries in order to attack our problem.
In particular we first define the right spaces in which find the solutions.
Then we show rigorously that $\mathcal K$ 
is the fundamental solution of the operator $-\Delta+a^{2}\Delta^{2}$. Probably this result is known,
but we were not able to find it in the literature. In particular Lemma
\ref{lemfundsol} is interesting of its right.
Moreover the smooth energy functional $\mathcal J_{q}$ is 
defined in such a way that its critical points are exactly
solutions of \eqref{eq:ScBP} and its geometric properties  are proved.\\
In Section \ref{sec4} we prove Theorem \ref{th:existence}. Indeed the hypothesis of small 
charges allows to have the Mountain Pass Geometry for any value of $p\in(2,6)$. 
Here the boundedness of the (PS) sequences is obtained by means of the key Proposition
\ref{lem:Tbarra}. On the other hand
the compactness is recovered by means of the Splitting Lemma \ref{lemSplitting}.\\
Section \ref{sec:allq} is devoted to the proof of Theorem \ref{th:existence2}. In this case, that is for $p\in(3,6)$,
the Mountain Pass Geometry holds for every $q$.  However the boundedness of the (PS)
sequences is obtained in two different way: it is {\em standard} if $p\in[4,6)$ while it is obtained by means
of the monotonicity trick (see \cite{JJ, St}) if $p\in(3,4)$. Even now the compactness 
(and then the existence of a solution) can be recovered by using the Splitting Lemma; nevertheless,
just to use a different (and simpler) argument, we prefer to give the proof in the radial setting.\\
In Section \ref{sec:Limite} we study the behaviour of the radial solutions with respect to $a$.
After proving Lemma \ref{lem:ato0} which may be of some interest in other contexts,
we prove Theorem \ref{th:ato0}.\\
In Appendix \ref{app:A} we collect few facts concerning the regularity of our solutions,
we prove the Poho\v zaev identity and give the proof of some nonexistence results in
the cases $p\geq 6$ and $p\leq 2$. In particular, to achieve our goal, in the case $p\leq 2$ we use an interesting identity, obtained in Section \ref{id}, which could be useful for further developments.\\
Finally in Appendix \ref{appsplitt} we give the proof of the Splitting Lemma \ref{lemSplitting}.

\subsection*{Notations}
As a matter of notations through the paper 
\begin{itemize}
\item we denote with $\|\cdot\|_{p}$ the usual $L^{p}(\R^3)$ norm;
\item $\langle \cdot, \cdot \rangle $ is the scalar product in $H^{1}(\mathbb R^{3})$
which gives rise to the  norm $\|\cdot\|=\sqrt{\|\nabla\cdot\|_2^2+\omega\|\cdot\|_2^2}$;
\item $p'=p/(p-1)$ is the conjugate exponent of $p$;
\item we use the symbol $o_{n}(1)$ for a vanishing sequence in the specified space;
\item we use $C,C_{1},C_{2},\ldots$ to denote suitable positive constants whose value may also change from line to line;
\item if not specified, the domain of the integrals is $\R^3$.
\end{itemize}
Other notations will be introduced whenever we need. Moreover, for simplicity, from now on we will consider positive $q$'s.


\section{Deduction of the Schr\"odinger-Bopp-Podolsky system}\label{deduction}

Let us consider the nonlinear Schr\"odinger Lagrangian density
\[
\mathcal{L}_{\rm Sc}
=i\hbar \bar{\psi} \partial_t \psi
-\frac{\hbar^2}{2m} |\nabla\psi|^2 + \frac{2}{p} |\psi|^p,
\]
where $\psi:\R\times\R^3\to\mathbb{C}$, $\hbar,m,p>0$, and let $(\phi,{\bf A})$ be the gauge potential of the electromagnetic field $({\bf E},{\bf H})$, namely $\phi:\R^3\to\R$ and ${\bf A}:\R^3\to\R^3$ satisfy
\[
\mathbf{E}=-\nabla\phi-\frac{1}{c}\partial_t\mathbf{A},
\qquad
\mathbf{H}=\nabla\times\mathbf{A}.
\]
The coupling of the field $\psi$ with the electromagnetic field $({\bf E},{\bf H})$ through the minimal coupling rule, namely the study of the interaction between $\psi$ and its own electromagnetic field, can be obtained replacing in $\mathcal{L}_{\rm Sc}$ the derivatives $\partial_t$ and $\nabla$ respectively with the covariant ones 
$$ 
D_{t}= \partial_{t}+\frac{iq}{\hbar} \phi,
\qquad
\mathbf D=\nabla -\frac{iq}{\hbar c}  \mathbf A,
$$
$q$ being a {\em coupling} constant.
This leads to consider
\begin{align*}
\mathcal{L}_{\rm CSc}
&=
i \hbar \overline \psi  D_{t}\psi -\frac{\hbar^{2}}{2m} |\mathbf D \psi|^{2}+\frac{2}{p} |\psi|^p\\
&=
i\hbar \overline{\psi} \left(\partial_t +\frac{iq}{\hbar} \phi\right) \psi 
-\frac{\hbar^2}{2m} \left|\left(\nabla- \frac{iq}{\hbar c}{\bf A}\right)\psi\right|^2 + \frac{2}{p} |\psi|^p.
\end{align*}
Now, to get the total Lagrangian density, we have to add to $\mathcal{L}_{\rm CSc}$ the Lagrangian density of the electromagnetic field.\\
The Bopp-Podolsky Lagrangian density (see \cite[Formula (3.9)]{Pob42}) is
\begin{align*}
\mathcal{L}_{\rm BP}
&=
\frac{1}{8\pi}\left\{
|\mathbf{E}|^2 - |\mathbf{H}|^2 + a^2 \left[ (\operatorname{div} \mathbf{E})^2 - \left|\nabla\times\mathbf{H} - \frac{1}{c} \partial_t\mathbf{E}\right|^2\right]
\right\}\\
&=
\frac{1}{8\pi}\left\{
|\nabla\phi+\frac{1}{c}\partial_t\mathbf{A}|^2 - |\nabla\times\mathbf{A}|^2
\right.\\
&\qquad\qquad
\left.+ a^2 \left[ \left(\Delta \phi+\frac{1}{c}\operatorname{div} \partial_t \mathbf{A}\right)^2 - \left|\nabla\times\nabla\times\mathbf{A} + \frac{1}{c} \partial_t(\nabla\phi+\frac{1}{c}\partial_t\mathbf{A})\right|^2\right]
\right\}.
\end{align*}
Thus the total action is
\[
\mathcal{S}(\psi,\phi,\mathbf{A})=\int \mathcal{L} dxdt
\]
where $\mathcal{L}:=\mathcal{L}_{\rm CSc} + \mathcal{L}_{\rm BP}$ is the total Lagrangian density.\\
The Euler-Lagrange equations of $\mathcal{S}$ are given by
\[
\begin{cases}
\displaystyle
i\hbar \left(\partial_t +\frac{iq}{\hbar} \phi \right)\psi
+\frac{\hbar^2}{2m} \left(\n- \frac{iq}{\hbar c}{\bf A}\right)^2\psi
+  |\psi|^{p-2}\psi 
=0 \smallskip\\
\displaystyle
-\operatorname{div}\left(\nabla\phi+\frac{1}{c}\partial_{t}\mathbf{A}\right)  
+ a^2
\left[ \Delta\left(\Delta \phi+\frac{1}{c}\operatorname{div} \partial_{t}\mathbf{A}\right)
- \frac{1}{c}\partial_t\operatorname{div}\left(\nabla\times\nabla\times\mathbf{A} + \frac{1}{c} \partial_t(\nabla\phi+\frac{1}{c}\partial_{t}\mathbf{A})\right) \right]\\
\qquad\qquad
=4\pi q |\psi|^2  \\
\displaystyle
-\frac{\hbar q}{mc}\Im \left[\left(\nabla\bar{\psi}  + \frac{iq}{\hbar c}{\bf A}\bar{\psi}  \right) \psi \right]
-\frac{1}{4\pi}\left\{
\frac{1}{c} \partial_t (\nabla\phi+\frac{1}{c}\partial_t{\bf A} )  
+ \nabla\times \nabla\times{\bf A}  \right\}\smallskip\\
\displaystyle
\qquad \qquad
+ \frac{a^2}{4\pi} \left[ \frac{1}{c} \nabla \partial_t\left(\Delta \phi+\frac{1}{c}\operatorname{div} \partial_t {\bf A} \right)
- \nabla\times\nabla\times\nabla\times\nabla\times{\bf A} 
- \frac{1}{c^2} \partial_{tt}\nabla\times\nabla\times{\bf A}
\right.\\
\displaystyle
\qquad \qquad
\left.
- \frac{1}{c} \nabla\times\nabla\times \partial_t(\nabla\phi+\frac{1}{c}\partial_t{\bf A} )
- \frac{1}{c^3} \partial_{ttt}(\nabla\phi+\frac{1}{c}\partial_t{\bf A} )
\right]=0.
\end{cases}
\]
If we consider $\psi(t,x)=e^{iS(t,x)}u(t,x)$  with $S,u:\R\times\R^3\to\R$ the Euler-Lagrange equations are
	\[
	\begin{cases}
	\displaystyle
	-\frac{\hbar^2}{2m} \Delta u
	+\left[ \frac{\hbar^2}{2m} \left| \nabla S - \frac{q}{\hbar c}{\bf A} \right|^2  + \hbar \partial_{t}S + q \phi \right]u
	= |u|^{p-2}u\smallskip\\ 
	\displaystyle \partial_{t}   u^2 
	+\frac{\hbar}{m} \operatorname{div}\left[\left(\nabla S - \frac{q}{\hbar c}{\bf A} \right) u^2 \right]=0\smallskip\\
	\displaystyle
	-\operatorname{div}\left(\nabla\phi+\frac{1}{c}\partial_{t}\mathbf{A}\right)  
	+ a^2
	\left[ \Delta\left(\Delta \phi+\frac{1}{c}\operatorname{div} \partial_{t}\mathbf{A}\right)
	- \frac{1}{c}\partial_t\operatorname{div}\left(\nabla\times\nabla\times\mathbf{A} + \frac{1}{c} \partial_t(\nabla\phi+\frac{1}{c}\partial_{t}\mathbf{A})\right) \right]\smallskip\\
	\qquad \qquad
	=4\pi q |u|^2  \\
	\displaystyle
	\frac{\hbar q}{mc} \left(\nabla S - \frac{q}{\hbar c}{\bf A}\right)u^2
	-\frac{1}{4\pi}\left\{
	\frac{1}{c} \partial_t (\nabla\phi+\frac{1}{c}\partial_t{\bf A} )  
	+ \nabla\times \nabla\times{\bf A}  \right\}\smallskip\\
	\displaystyle
	\qquad \qquad
	+ \frac{a^2}{4\pi} \left[ \frac{1}{c} \nabla \partial_t\left(\Delta \phi+\frac{1}{c}\operatorname{div} \partial_t {\bf A} \right)
	- \nabla\times\nabla\times\nabla\times\nabla\times{\bf A} 
	- \frac{1}{c^2} \partial_{tt}\nabla\times\nabla\times{\bf A}
	\right.\smallskip\\
	\displaystyle
	\qquad \qquad
	\left.
	- \frac{1}{c} \nabla\times\nabla\times \partial_t(\nabla\phi+\frac{1}{c}\partial_t{\bf A} )
	- \frac{1}{c^3} \partial_{ttt}(\nabla\phi+\frac{1}{c}\partial_t{\bf A} )
	\right]=0.
	\end{cases}
	\]
Finally, if we consider standing waves $\psi(t,x) = e^{i\omega t/\hbar} u(x)$ in the purely electrostatic case ($\phi=\phi(x)$ and ${\bf A}={\bf 0}$), the  second and fourth equation are satisfied and we get
\begin{equation}
\label{SBPvero}
\tag{$\mathcal{SBP}$}
\begin{cases}
\displaystyle
-\frac{\hbar^2}{2m} \Delta u + \omega u + q \phi u
= |u|^{p-2}u\\ 
-\Delta\phi
+ a^2 \Delta^2 \phi
=4\pi q u^2 . 
\end{cases}
\end{equation}
{\em Normalising} the constants $\hbar$ and $m$
and renaming the unknown $\phi$ it is easy to see that solutions of \eqref{eq:ScBP} give rise to solutions of \eqref{SBPvero}.
Hence from now on we  will refer to system \eqref{eq:ScBP}.

\section{Preliminaries}\label{sec:preliminaries}
In this section we give some preliminary results that will be useful for our arguments. In particular we give some fundamental properties on the operator $-\Delta+a^2\Delta^2$. Then we introduce the functional whose critical points are weak solutions of our problem and we conclude the section showing that, at least for small $q$'s, such a functional satisfies the geometrical assumptions of the Mountain Pass Theorem.

\subsection{The operator $-\Delta+a^2\Delta^2$}\label{secop}

Let $\mathcal{D}$ be the completion of $C_c^\infty(\R^3)$ with respect to the norm 
$\|\cdot\|_{\mathcal D}$ induced by the scalar product
\[
\langle \varphi ,\psi \rangle_{\mathcal D} := \int \nabla \varphi \nabla \psi + a^2 \int \Delta \varphi \Delta \psi.
\]
Then $\mathcal{D}$ is an Hilbert space continuously embedded into $D^{1,2}(\RT)$ and 
consequently  in $L^6(\RT)$.\\
It is interesting to note also the following result.
\begin{Lem}\label{lem31}
The space $\mathcal{D}$ is continuously embedded in $L^\infty(\R^3)$.
\end{Lem}
\begin{proof}
Let $\varphi\in C_c^\infty (\R^3)$, $x\in\R^3$, and $Q$ be a unitary cube containing $x$. Arguing as in Brezis \cite[Proof of Theorem 9.12]{Brezis}, using the Sobolev inequality applied to $\varphi$ and to $\nabla\varphi$, and since
\[
\sum_{i,j}\int \partial_{ij} \varphi \ \partial_{ij} \varphi 
= \int \Delta \varphi\  \Delta \varphi,
\]
we have
\[
|\varphi(x)|
\leq |\overline{\varphi}|+C \|\nabla \varphi\|_{L^6(Q)}
\leq C \|\varphi\|_{W^{1,6}(Q)}
\leq C \|\varphi\|_{W^{1,6}(\R^3)}
\leq C (\|\nabla \varphi \|_2 + \|\Delta \varphi\|_2)
\leq C \|\varphi\|_{\mathcal{D}}.
\]
Here $\overline{\varphi}$ is the mean of $\varphi$ on $Q$ and $C$'s do not depend on $Q$ and $\varphi$.
Therefore, standard density arguments allow to conclude.
\end{proof}

The next Lemma gives a useful characterization of the space $\mathcal D$.
\begin{lemma}
The space $C^{\infty}_{c}(\mathbb R^{3})$ is dense in
$$\mathcal A := \left\{\phi \in D^{1,2}(\mathbb R^{3}) : \Delta \phi \in L^{2}(\mathbb R^{3})\right\}$$
normed by $\sqrt{\langle \phi,\phi\rangle_{\mathcal D}}$ and, therefore, $\mathcal D=\mathcal A$.
\end{lemma}
\begin{proof}
Let $\phi\in \mathcal A$, $\rho\in C^{\infty}_{c}(\mathbb R^{3};\mathbb R_{+})$, $\|\rho\|_{1}=1$, and $\{\rho_{n}\}\subset C^{\infty}_{c}(\mathbb R^{3})$ the sequence of mollifiers given by
$\rho_{n}(x) = n^{3}\rho(nx)$. Define  $\phi_{n} : =\rho_{n}*\phi \in C^{\infty}(\mathbb R^{3})$.
Since, recalling the well known properties of the mollifiers, 
$$\partial_{i}\phi_{n}=\rho_{n}*\partial_{i}\phi \in L^{2}(\mathbb R^{3}),
\ \
i=1,2,3,
\quad
 \Delta\phi_{n} =\rho_{n}*\Delta\phi \in L^{2}(\mathbb R^{3}),
 $$
and
$$\| \nabla \phi_n - \nabla \phi\|_2 \to 0, \quad \|\Delta \phi_n -\Delta \phi\|_2 \to 0,$$
 we have
\begin{equation}\label{eq:Cinfinito}
\phi_{n} \in C^{\infty}(\mathbb R^{3})\cap\mathcal A\ \ \text{ and }\ \ \| \phi_{n} - \phi\|_{\mathcal D}
\to0.
\end{equation}
Let now $\xi\in C^{\infty}(\mathbb R^{3})\cap \mathcal A$,
$\zeta\in C_{c}^{\infty}(\mathbb R^{3}; [0,1])$ with $\zeta(x) = 1$ in $B(0,1)$, $\supp(\zeta)\subset B(0,2)$
and define
\begin{equation*}
\xi_{n}:=\zeta(\cdot/n) \xi \in C_{c}^{\infty}(\mathbb R^{3}).
\end{equation*}
We have
\begin{align*}
\nabla \xi_{n} 
&=
\zeta(\cdot/n)\nabla\xi+\frac{1}{n} \xi \nabla\zeta(\cdot/n),\\
\Delta \xi_{n}
&=
\zeta(\cdot/n)\Delta \xi + \frac2n\nabla \xi \nabla \zeta(\cdot/n) + \frac{1}{n^{2}}\xi\Delta \zeta(\cdot/n).
\end{align*}
Noticing that
\[
\frac{1}{n^2} \int \xi^2(x) \left|\nabla\zeta \left(\frac{x}{n}\right)\right|^2
\leq
\frac{1}{n^2}
\left(\int_{|x|\geq n}   \xi^6\right)^{1/3} 
\left(\int \left|\nabla \zeta \left(\frac{x}{n}\right)\right|^3\right)^{2/3}
=
C \left(\int_{|x|\geq n}   \xi^6\right)^{1/3} \to 0
\]
and, analogously,
\[
\frac2n\nabla \xi \nabla \zeta(\cdot/n),
\frac{1}{n^{2}}\xi\Delta \zeta(\cdot/n)
\to 0
\quad  \text{ in } L^{2}(\mathbb R^{3}),
\]
as $n\to+\infty$, we infer
\begin{align*}
\| \nabla \xi -\nabla \xi_{n}\|_{2}^{2}
&\leq
2\| (1-\zeta(\cdot/n))\partial_{i}\xi\|_{2}^{2} +o_{n}(1)=o_{n}(1) \\
\|\Delta \xi -\Delta \xi_{n}\|_{2}^{2}
&\leq
2\| (1-\zeta(\cdot/n))\Delta \xi\|_{2}^{2} +o_{n}(1)=o_{n}(1)
\end{align*}
showing that $\|\xi_{n} - \xi\|_{\mathcal D}\to 0$. This joint with \eqref{eq:Cinfinito} concludes the proof.
\end{proof}

For every fixed $u\in H^1(\R^3)$, the Riesz Theorem implies that there exists a unique solution $\phi_u\in\mathcal{D}$ of the second equation in \eqref{eq:ScBP}. 
To write {\em explicitly} such a solution (see also \cite[Formula (2.6)]{Pob42}), we consider
\[
\mathcal{K}(x)=\frac{1-e^{-|x|/a}}{|x|}.
\]
We have the following fundamental properties.
\begin{Lem}\label{lemfundsol}
For all $y\in\R^3$, $\mathcal{K}(\cdot-y)$ solves in the sense of distributions
\[
-\Delta \phi +a^2\Delta^2 \phi = 4\pi\delta_y.
\]
Moreover
\begin{enumerate}[label=(\roman*),ref=\roman*]
	\item \label{ifundsol} if $f\in L^1_{\rm loc}(\R^3)$ and, for a.e. $x\in\R^3$, the map $y\in\R^3\mapsto f(y)/|x-y|$ is summable, then $\mathcal{K}*f \in L^1_{\rm loc}(\R^3)$;
	\item \label{iifundsol}if $f\in L^p (\R^3)$ with $1\leq p< 3/2$, then $\mathcal{K}*f\in L^q(\R^3)$ for $q\in(3p/(3-2p),+\infty]$.
\end{enumerate}
In both cases $\mathcal{K}*f$ solves
\begin{equation}
\label{fundsolf}
-\Delta \phi +a^2\Delta^2 \phi = 4\pi f
\end{equation}
in the sense of distributions and we have the following distributional derivatives
\[
\nabla (\mathcal{K}*f)= (\nabla\mathcal{K})*f
\quad\hbox{and}\quad
\Delta (\mathcal{K}*f)= (\Delta\mathcal{K})*f
\quad\hbox{a.e. in } \R^3.
\]
\end{Lem}
\begin{proof}
Let us consider for simplicity $y=0$ and prove that for every $\varphi\in C_c^\infty(\R^3)$
\begin{equation*}
-\int \mathcal{K} \Delta \varphi +a^2\int \mathcal{K}  \Delta^2 \varphi = 4\pi \varphi(0).
\end{equation*}
Of course it is enough to show that
\begin{equation}
\label{limitr0}
\lim_{r\to 0^+} I(r)= 4\pi \varphi(0)
\end{equation}
where
\[
I(r):=-\int_{|x|>r} \mathcal{K} \Delta \varphi +a^2\int_{|x|>r} \mathcal{K}  \Delta^2 \varphi.
\]
Since $\varphi$ has compact support, we consider the annulus $A:=\{x\in\R^3: r<|x|<R\}$ for $R$ large enough and a standard integration by parts shows that
\begin{align*}
I(r)
&=
-\int_A \mathcal{K} \Delta \varphi +a^2\int_A \mathcal{K}  \Delta^2 \varphi\\
&=
 \int_A \varphi (-\Delta \mathcal{K} +a^2\Delta^2  \mathcal{K})
+ \int_{|x|=r} \varphi (\nabla\mathcal{K} - a^2 \nabla \Delta \mathcal{K}  )\cdot \nu
 +\int_{|x|=r} \mathcal{K} (a^2 \nabla(\Delta\varphi)-\nabla\varphi )\cdot \nu\\
 &\quad
+ a^2 \int_{|x|=r} \Delta \mathcal{K} \nabla\varphi \cdot \nu
-a^2 \int_{|x|=r}  \Delta\varphi \nabla\mathcal{K}  \cdot \nu
\end{align*}
where $\nu$ is the unit outward normal to $A$.
Since $\varphi$ is continuous, $(a^2 \nabla(\Delta\varphi)-\nabla\varphi )\cdot \nu,  \nabla\varphi \cdot \nu$, and $\Delta\varphi$ are bounded,
$\mathcal{K}$ can be extended continuously in $0$ by setting $\mathcal K(0)=1/a$, and, for $x\neq0$,	
\begin{align}
\nabla\mathcal{K}&=-\frac{x}{|x|^3}+\frac{x}{|x|^3}\left(\frac{|x|}{a}+1\right)e^{-|x|/a},
\label{eq:nablaK}
\\
\Delta\mathcal{K}&=-\frac{e^{-|x|/a}}{a^2 |x|},
\label{eq:DeltaK}\\
\nabla\Delta\mathcal{K}&=\frac{x}{a^2 |x|^3}\left(\frac{|x|}{a}+1\right)e^{-|x|/a},
\nonumber\\
\Delta^2 \mathcal{K}&= \frac{1}{a^2}\left(\Delta \mathcal{K} + \operatorname{div}\frac{x}{|x|^3}\right), \nonumber
\end{align}
we have
\begin{align*}
-\Delta \mathcal{K} +a^2\Delta^2  \mathcal{K}= 0
\hbox{ in }A
& \Longrightarrow
\int_A \varphi (-\Delta \mathcal{K} +a^2\Delta^2  \mathcal{K})=0,\\
(\nabla\mathcal{K} - a^2 \nabla \Delta \mathcal{K}) \cdot \nu = \frac{1}{r^2}
\hbox{ on } |x|=r
& \Longrightarrow
\int_{|x|=r} \varphi (\nabla\mathcal{K} - a^2 \nabla \Delta \mathcal{K}  )\cdot \nu\\
&\qquad
= \int_S \varphi(r\sigma) d\sigma \to 4\pi \varphi(0)
\hbox{ as } r\to 0^+,\\
\mathcal{K} \leq 1/a
& \Longrightarrow
\int_{|x|=r} \mathcal{K} (a^2 \nabla(\Delta\varphi)-\nabla\varphi )\cdot \nu
\to 0
\hbox{ as } r\to 0^+,\\
|\Delta\mathcal{K}| \leq 1/(a^2|x|)
&\Longrightarrow
\int_{|x|=r} \Delta \mathcal{K} \nabla\varphi \cdot \nu
\to 0
\hbox{ as } r\to 0^+,\\
\nabla\mathcal{K}  \cdot \nu = \frac{1}{r^2} \left[1-\left(\frac{r}{a}+1\right)e^{-\frac{r}{a}}\right]
\hbox{ on } |x|=r
&\Longrightarrow
\int_{|x|=r}  \Delta\varphi \nabla\mathcal{K}  \cdot \nu
\to 0
\hbox{ as } r\to 0^+,
\end{align*}
where $S$ is the unit sphere in $\R^3$. Thus \eqref{limitr0} is proved.\\
To get (\ref{ifundsol}), we observe that, by Fubini Theorem, for all balls $B\subset\R^3$ 
\[
\int_B |\mathcal{K}*f|
\leq
\int_{\R^3} \left(\int_B \mathcal{K}(x-y)dx\right) |f(y)| dy
\]
and, since $\mathcal{K}\leq | \cdot|^{-1}$, we can conclude arguing as in \cite[Proof of Theorem 6.21]{LL}.\\
Since $\mathcal K \in L^{\tau}(\mathbb R^{3})$ for $\tau\in(3,+\infty]$, the Young inequality (see e.g. \cite[Inequality (4), p 99]{LL}) allows to get (\ref{iifundsol}), since
\[
\frac{1}{q}=\frac{1}{p}+\frac{1}{\tau}-1 < \frac{1}{p}+\frac{1}{3}-1 = \frac{3-2p}{3p}.
\]
Moreover the fact that $\mathcal{K}*f$ solves \eqref{fundsolf} in the sense of distributions, namely, for all $\varphi\in C_c^\infty(\R^3)$
\[
-\int (\mathcal{K}*f) \Delta\varphi + a^2 \int (\mathcal{K}*f) \Delta^2 \varphi = 4\pi \int f \varphi,
\]
is a consequence of the Fubini Theorem and of the first part of this Lemma.\\
To conclude, let us consider, for instance, the assumptions in (\ref{ifundsol}). The proof of the remaining case is similar.\\
We claim that the functions $(\nabla\mathcal{K})*f$ and $(\Delta\mathcal{K})*f$ are well defined a.e. in $\R^3$. In fact, by \eqref{eq:nablaK},
we have that, for every $i=1,2,3$,
\begin{equation}
\label{gradKbound}
|\partial_i  \mathcal{K}(x)|
\leq \frac{1}{|x|^2}\left(1 - e^{-|x|/a} - \frac{|x|}{a}e^{-|x|/a}\right)
\leq C
\quad
\hbox{for } |x| \hbox{ small}
\end{equation}
and, since
\begin{equation}
\label{38bis}
|\partial_i  \mathcal{K}(x)|\leq\frac{1}{|x|^2}+\frac{1}{|x|^2}\left(\frac{|x|}{a}+1\right)e^{-|x|/a},
\end{equation}
then, in particular,
\[
|\partial_i  \mathcal{K}(x)|\leq \frac{2}{|x|^2} +\frac{1}{a|x|}\leq \frac{C}{|x|}
\quad
\hbox{for } |x| \hbox{ large} .
\]
Thus, if $r>0$ is sufficiently large, for a.e. $x$ in $\R^3$, using the summability of the map $y\in\R^3\mapsto f(y)/|x-y|$, 
we deduce
\begin{align*}
|[\partial_i  \mathcal{K}* f ](x)|
&\leq
\int_{B_r(x)} |\partial_i  \mathcal{K}(x-y) | |f(y)|dy
+  \int_{B_r^c(x)} |\partial_i  \mathcal{K} (x-y)| |f(y)|dy\\
&\leq
C \left[
\int_{B_r(x)} |f(y)| dy +  \int_{B_r^c(x)} \frac{|f(y)|}{|x-y|}dy
\right] <+\infty.
\end{align*}
Moreover, by \eqref{eq:DeltaK}, $|\Delta\mathcal{K}|\leq 1/(a^2|\cdot|)$ and so, arguing again as in \cite[Proof of Theorem 6.21]{LL} we get the claim.\\ 
Then, since
\[
|\mathcal{K}(x)|, |\partial_i  \mathcal{K}(x)|, |\Delta  \mathcal{K}(x)|\leq \frac{C}{|x|},
\]
and the map $y\in\R^3\mapsto f(y)/|x-y|$ is summable, \cite[Theorem 6.21]{LL} implies that
\[
\mathcal{K}*f, \partial_{i}\mathcal{K}*f,\Delta\mathcal{K}*f\in L_{\rm loc}^1(\R^3)
\]
and so, for all $\varphi\in C_c^\infty(\R^3)$,
\begin{align*}
(x,y)&\mapsto \partial_i \varphi (x) \mathcal{K}(x-y) f(y) \in L^1(\R^3\times\R^3),\\
(x,y)&\mapsto \Delta \varphi (x) \mathcal{K}(x-y) f(y) \in L^1(\R^3\times\R^3),\\
(x,y)&\mapsto \varphi (x) \partial_i  \mathcal{K} (x-y) f(y)\in L^1(\R^3\times\R^3),\\
(x,y)&\mapsto \varphi (x) \Delta  \mathcal{K} (x-y) f(y)\in L^1(\R^3\times\R^3).
\end{align*}
Hence, by Fubini's Theorem and using a limit argument as in the first part of this proof we have that for all $\varphi\in C_c^\infty(\R^3)$
\[
\int (\mathcal{K}*f) \partial_i \varphi  = - \int \varphi (\partial_i \mathcal{K})*f
\qquad i=1,2,3
\]
and 
\[
\int (\mathcal{K}*f) \Delta\varphi  = \int \varphi (\Delta \mathcal{K})*f.
\]
The proof is thereby completed.
\end{proof}
Then, if we fix $u\in H^1(\R^3)$, the unique solution in $\mathcal{D}$ 
of the second equation in \eqref{eq:ScBP} is
\begin{equation}
\label{phiu}
\phi_u:=\mathcal{K}*u^2.
\end{equation}
Actually the following useful properties hold.
\begin{lemma}\label{lem:propphi}
	For every $u\in H^{1}(\mathbb R^{3})$ we have:
	\begin{enumerate}[label=(\roman*),ref=\roman*]
		\item\label{propphiii} for every $y\in\R^3$, $\phi_{u( \cdot+y)} = \phi_{u}( \cdot+y)$;
		\item\label{propphiiii} $\phi_{u}\geq0$;
		\item \label{propphiiv} for every  $s\in (3,+\infty]$, $\phi_{u}\in   L^{s}(\mathbb R^{3})\cap C_{0}(\mathbb R^{3})$;
		\item \label{propphiv} for every $s\in (3/2,+\infty]$, $\nabla \phi_{u} = \nabla \mathcal K * u^{2}\in L^{s}(\mathbb R^{3})\cap C_{0}(\mathbb R^{3})$;
		\item \label{propphivi}  $\phi_u\in\mathcal{D}$;
		\item \label{propphivii} $\|\phi_{u}\|_{6}\leq C \|u\|^{2}$;
		\item \label{propphiviii} $\phi_{u}$ is the unique minimizer of the functional
		$$E(\phi) = \frac12 \|\nabla \phi\|_{2}^{2} +\frac {a^{2}}{2} \|\Delta\phi\|_{2}^{2}- \int\phi u^{2}, \quad \phi\in \mathcal D.$$
	\end{enumerate}
Moreover
\begin{enumerate}[label=(\roman*),ref=\roman*]\setcounter{enumi}{7}
	\item \label{propphii} if $v_{n}\rightharpoonup v$ in $H^{1}(\mathbb R^{3})$, then $\phi_{v_{n}} \rightharpoonup \phi_{v}$ in $\mathcal D$.
\end{enumerate}
\end{lemma}
\begin{proof}
	Let us fix $u\in H^{1}(\mathbb R^{3})$.
	Items  \eqref{propphiii} and \eqref{propphiiii} are obvious.\\
	Being
	$\mathcal K \in L^{\tau}(\mathbb R^{3})$ for $\tau\in(3,+\infty]$, by well known properties of the convolution product, we have (\ref{propphiiv}).\\
	Moreover, since $\nabla \mathcal K$ is bounded near $0$ (see \eqref{gradKbound}) and, by \eqref{38bis},
	decays as $|\cdot|^{-2}$ at infinity, then $\nabla \mathcal K\in L^{\tau}(\mathbb R^{3})$ for $\tau\in(3/2,+\infty]$ and so we get (\ref{propphiv}).\\
	Property (\ref{propphivi}) holds since $\nabla\phi_u\in L^{2}(\mathbb R^{3})$ and $\Delta\phi_u =\Delta\mathcal{K}*u^2\leq 1/(a^2|\cdot|)*u^2 \in L^{2}(\mathbb R^{3})$.\\
	Multiplying the second equation in \eqref{eq:ScBP} by the solution $\phi_{u}$, integrating and using Lemma \ref{lem31}, we find
	\begin{equation*}
	\|\phi_{u}\|_{\mathcal D}^{2}\leq C\|u^{2}\|_{1}\|\phi_{u}\|_{\infty}\leq C\|u\|^{2}\|\phi_{u}\|_{\mathcal D}
	\end{equation*}
	and then
	\begin{equation*}
	\|\phi_{u}\|_{6}\leq C \|\phi_{u}\|_{\mathcal D}\leq C \|u\|^{2}.
	\end{equation*}
	obtaining \eqref{propphivii}. Property \eqref{propphiviii} is also trivial.\\	
	We conclude observing that for every $\varphi\in  C_{c}^{\infty}(\mathbb R^{3})$ we have
	$$\langle \phi_{v_{n}} , \varphi\rangle_{\mathcal D} =4\pi \int v_{n}^{2}\varphi \to 4\pi \int v^{2} \varphi = 
	\langle \phi_{v} , \varphi\rangle_{\mathcal D} $$
	and we obtain \eqref{propphii} by density.	
\end{proof}

\subsection{The functional setting}
It is easy to see that the critical points of the $C^{1}$ functional
\begin{equation*}
 F_{q}(u,\phi) = \frac{1}{2} \|\nabla u\|_2^2 + \frac{\omega}{2}  \|u\|_2^2 + \frac{q^{2}}{2} \int \phi u^2
-\frac{q^{2}}{16\pi} \|\nabla \phi\|_2^2 -\frac{a^2q^{2}}{16\pi} \|\Delta \phi\|_2^2 
-\frac{1}{p} \|u\|_p^p
 \end{equation*}
on $H^{1}(\mathbb R^{3})\times \mathcal D $ are weak solutions of \eqref{eq:ScBP}.
Indeed if $ (u, \phi)\in H^{1}(\mathbb R^{3})\times \mathcal D$ is a critical point of $F_{q}$ then
\begin{equation*}
0=\partial_{u} F_{q}( u,\phi)[v]
=\int \nabla u\nabla v+\omega\int uv
+q^{2}\int \phi u v -\int |u|^{p-2} u v, 
\qquad
\hbox{for all } v\in H^{1}(\mathbb R^{3}),
\end{equation*}
and
\begin{equation*}
0=\partial_{\phi} F_{q}( u,\phi)[\xi]=\frac{ q^{2}} {2} \int {u}^{2} \xi 
- \frac{q^{2}}{8\pi}\int\nabla \phi \nabla \xi
- \frac{a^2 q^{2}}{8\pi}\int\Delta \phi \Delta \xi,
\qquad
\hbox{for all } \xi\in \mathcal D. 
\end{equation*}
However the functional $F_{q}$ is strongly unbounded from below and above and
hence the usual techniques of the Critical Point Theory cannot be used immediately.
Hence we adopt a reduction procedure which is successfully used also
with  other system of equations involving the coupling between
matter and electromagnetic field. Here we just revise the main argument. We refer the  reader to \cite{BF98,BFKleinG} for the details.
Then, first of all one observes that actually $\partial_{\phi} F_{q}$ is a $C^{1}$ function. Thus,
if $G_{\Phi}$ is the graph of the map 
$\Phi : u\in H^{1}(\mathbb R^{3})\mapsto \phi_{u}\in \mathcal D$, an application of the Implicit Function Theorem
gives
$$G_{\Phi}=\left\{(u,\phi)\in H^{1}(\mathbb R^{3})\times \mathcal D: \partial_{\phi} F_{q}(u,\phi) = 0\right\}
\quad  \text{ and } \quad \Phi \in C^{1}(H^{1}(\mathbb R^{3}); \mathcal D ).$$
Then we define the {\sl reduced} functional
\[
\mathcal J_{q}(u)
:=
F_{q}(u,\Phi(u))
=
\frac{1}{2} \|\nabla u\|_2^2 + \frac{\omega}{2}  \|u\|_2^2 + \frac{q^2}{4} \int \phi_u u^2
-\frac{1}{p} \|u\|_p^p,
\]
which is of class $C^{1}$ on $H^{1}(\mathbb R^{3})$  and, for all $u,v\in H^{1}(\mathbb R^{3})$
\begin{align*}
\mathcal J_{q}'(u)[v]
&=
\partial_{u} F_{q}(u,\Phi(u))[v]+
\partial_{\phi}F_{q}(u,\Phi(u))\circ \Phi'(u) [v]  \\
&=
\partial_{u} F_{q}(u,\Phi(u))[v] \\
&=
\int \nabla u \nabla v+\omega \int uv+q^{2}\int \phi_{u}uv - \int|u|^{p-2}uv.
\end{align*}

Rigorously the functional $\mathcal J_{q}$ should depend also on $a>0$.
However, for the sake of simplicity, we do not write explicitly this dependence,
which is deserved in Section \ref{sec:Limite} where the limit as $a\to 0$ is considered.

We have that the following statements are equivalent:
\begin{enumerate}[label=(\roman*),ref=\roman*]
	\item the pair $(u,\phi)\in H^{1}(\mathbb R^{3})\times \mathcal D$ is a critical point of $F_{q}$, 
	i.e. $(u,\phi)$ is a solution of \eqref{eq:ScBP};
	\item $u$ is a critical point of $\mathcal J_{q}$ and $\phi =\phi_{u}$.
\end{enumerate}

In virtue of this, to solve problem \eqref{eq:ScBP} is equivalent
to find the critical points of $\mathcal J_{q}$, namely to solve
\begin{equation*}
- \Delta u +\omega u+q^{2}\phi_{u} u = |u|^{p-2}u \quad  \mbox{in } 
\mathbb{R}^{3}.
\end{equation*}

%
%

\subsection{The Mountain Pass Geometry}
We conclude this section showing that the functional $\J_q$ satisfies the geometrical assumptions of the Mountain Pass Theorem \cite{AR}.

\begin{Lem}\label{Lem:MP}
	The functional $\J_q$ satisfies:
	\begin{enumerate}[label=(\roman*),ref=\roman*]
		\item \label{MPi}$\J_q(0)=0$;
		\item \label{MPii}there exist $\delta,\rho>0$ such that $\J_q(u)>\delta$ for all $u\in H^{1}(\R^3)$  with $\|u\|=\rho$;
		\item \label{MPiii}there exists $w\in H^{1}(\R^3)$ with $\|w\|>\rho$ such that $\J_q(w)<0$ for every $q\neq 0$ if $p\in(3,6)$, and for $q$ small enough if $p\in(2,3]$.
	\end{enumerate}
\end{Lem}
\begin{proof}
Property (\ref{MPi}) is trivial.\\
Moreover, since $\phi_u = \mathcal{K}*u^2\geq 0$, by Sobolev inequality we have
\[
\J(u)
\geq
C_1 \|u\|^2 - C_2 \|u\|^p
\]
and so, if we take $\rho>0$ small enough, we get (\ref{MPii}).\\
To prove (\ref{MPiii}), we  fix  $u\neq0$ in $H^{1}(\mathbb R^{3})$ and  distinguish two cases:\\
{\bf Case 1:} $p\in(3,6)$.\\	
Let $u_\tau=\tau^2 u(\tau \cdot)$. 
	We have
	\begin{align*}
	\J_q(u_\tau)
	&=
	\frac{\tau^{3}}{2}\|\nabla u\|_{2}^{2}
	+ \frac{\tau}{2}\omega\|u\|_{2}^{2}
	+q^2\frac{\tau^3 }{4}\iint \frac{1-e^{-\frac{|x -y|}{\tau a}}}{|x -y|}u^2(x)u^2(y)
	- \frac{\tau^{2p-3}}{p}\|u\|_{p}^{p}\\
	&\leq
	\frac{\tau^{3}}{2}\|\nabla u\|_{2}^{2}
	+ \frac{\tau}{2}\omega\|u\|_{2}^{2}
	+q^2\frac{\tau^3 }{4}\iint \frac{u^2(x)u^2(y)}{|x -y|}
	- \frac{\tau^{2p-3}}{p}\|u\|_{p}^{p}
	\end{align*}
	and the conclusion easily follows considering $\tau\to+\infty$. Actually, if $p\in(4,6)$, the simpler curve $u_\tau=\tau u$ also works.\\
{\bf Case 2:} $p\in(2,3]$.\\
Let $u_\tau=\tau^\frac{p}{p-2} u(\tau \cdot)$. We have 
\begin{equation}
\label{numero}
\J_q(u_\tau)
\leq
\frac{\tau^{\frac{p+2}{p-2}}}{2}\|\nabla u\|_{2}^{2}
+ \frac{\tau^{\frac{6-p}{p-2}}}{2}\omega\|u\|_{2}^{2}
+q^2\frac{\tau^{\frac{10-p}{p-2}} }{4} \iint \frac{u^2(x)u^2(y)}{|x -y|}
- \frac{\tau^{\frac{p^2-3p+6}{p-2}}}{p}\|u\|_{p}^{p}
\end{equation}
and since
	\[
	\frac{6-p}{p-2}
	<
	\frac{p+2}{p-2}
	<
	\frac{p^2-3p+6}{p-2}
	<
	\frac{10-p}{p-2},
	\]
	arguing as in \cite[Corollary 4.4]{RuizJFA} we have that, if $q=0$, the right hand side of \ref{numero} is unbounded from below (considering $\tau\to+\infty$) and thus, for $q$ small enough, its infimum is strictly negative and we conclude.
\end{proof}

\begin{Rem}\label{rem:seserve1}
We observe explicitly that $\delta, \rho $ and $w$ do not depend on $q$, neither on $a$;
indeed the term involving these two parameters has been successfully thrown away.
\end{Rem}

\section{Existence for small $q$'s in the case $p\in(2,6)$}\label{sec4}

In this section we prove an existence result for small $q$ and $p\in(2,6)$. Actually, as we will see in the next section, such a result will be improved in the case $p\in(3,6)$: we will be able to find solutions of \eqref{eq:ScBP} for all $q\neq 0$.

Let us consider, for every $T>0$,  the truncated functional
\[
\mathcal J_{q,T}(u) 
:= 
\frac{1}{2}\|\nabla u\|_{2}^{2} + \frac{\omega}{2}\|u\|_{2}^{2} +\frac{q^2}{4}K_T(u)\iint \frac{1-e^{-\frac{|x -y|}{a}}}{|x -y|} u^2(x)u^{2}(y) - \frac{1}{p}\|u\|_{p}^{p}
\]
where
\[
K_T(u):=\chi\left(\frac{\|u\|^2}{T^2}\right)
\]
and $\chi\in C_0^\infty(\R,\R)$ satisfies
\[
\chi(s):=
\begin{cases}
\chi(s)=1 & \hbox{ for } s\in[0,1],\\
0\leq\chi\leq 1 & \hbox{ for } s\in[1,2],\\
\chi(s)=0 & \hbox{ for } s\in[2,+\infty[,\\
\|\chi'\|_\infty \leq 2.
\end{cases}
\]
Observe that
\begin{align*}
\mathcal J'_{q,T}(u)[u]
&=
\|\nabla u\|_{2}^{2}
+ \omega\|u\|_{2}^{2} 
+q^2K_T(u)\iint \frac{1-e^{-\frac{|x -y|}{a}}}{|x -y|} u^2(x)u^{2}(y)\\
&\qquad
+\frac{q^2}{2T^2} \chi'\left(\frac{\|u\|^2}{T^2}\right) \|u\|^2 \iint \frac{1-e^{-\frac{|x -y|}{a}}}{|x -y|} u^2(x)u^{2}(y)
- \|u\|_{p}^{p}
\end{align*}

Arguing as in Lemma \ref{Lem:MP} we have
\begin{Lem}
	\label{lemMPGT}
	The functional $\mathcal{J}_{q,T}$ satisfies the geometric assumption of the Mountain Pass Theorem,
	namely:
	\begin{enumerate}[label=(\roman*),ref=\roman*]
		\item \label{MPTG1}$\mathcal{J}_{q,T}(0)=0$;
		\item \label{MPTG2} there exist $\delta,\rho>0$ such that for all $u\in H^1(\R^3)$ with $\|u\|=\rho$, $\mathcal J_{q,T}(u) \geq \delta$;
		\item \label{MPTG3} there exists $w\in H^1(\R^3)$ with $\|w\|>\rho$ such that $\mathcal J_{q,T}(w)<0$. 
	\end{enumerate}
\end{Lem}
\begin{proof}
	Property (\ref{MPTG1}) is trivial.\\
	By Sobolev inequality,
	\[
	\mathcal J_{q,T}(u) 
	\geq 
	\frac{1}{2}\|\nabla u\|_{2}^{2} 
	+ \frac{\omega}{2}\|u\|_{2}^{2} 
	- \frac{1}{p}\|u\|_{p}^{p}
	\geq
	\frac{1}{2}\| u\|^{2} 
	- C\|u\|^{p}
	\]
	an so, taking $\rho$ small enough, we get (\ref{MPTG2}).\\
	Finally, let us fix $\psi\in C_0^\infty(\R,\R)$ with $\|\psi\|=1$ and consider $\psi_t:= t\psi$ for $t>0$. If $t$ is sufficiently large, then
	\[
	\mathcal J_{q,T}(\psi_t) 
	= 
	\frac{t^2}{2}
	- \frac{t^p}{p}\|\psi\|_{p}^{p}<0
	\]
	and we get (\ref{MPTG3}).
\end{proof}

\begin{Rem}\label{rem:seserve2}
As in Remark \ref{rem:seserve1} we have that $\delta,\rho$ and $w$ do not depend on $q,a,T$.
\end{Rem}

In virtue of the above Lemma we can define the Mountain Pass level for $\mathcal J_{q,T}$, namely
\[
c_{q,T}:=\inf_{\gamma\in\Gamma_{q,T}}\max_{t\in[0,1]} \mathcal J_{q,T}(\gamma(t))>0,
\]
where $
\Gamma_{q,T}
:=
\left\{\gamma \in C([0,1],H^1(\R^3)):  \gamma(0)=0, \mathcal J_{q,T}(\gamma(1))<0 \right\}.
$
By the Ekeland Variational Principle 
(see also \cite{Willem}) there exists a (PS) sequence $\{u_n\}\subset H^1(\R^3)$ for $\mathcal J_{q,T}$ at level $c_{q,T}$. \\
We can define also
\[
c_{q}:=\inf_{\gamma\in\Gamma_q}\max_{t\in[0,1]} \mathcal J_{q}(\gamma(t))>0,
\]
$\Gamma_q
:=
\left\{\gamma \in C([0,1],H^1(\R^3)):  \gamma(0)=0, \mathcal J_{q}(\gamma(1))<0 \right\}$,
the Mountain Pass level associated to $\mathcal J_{q}$.
Since $\mathcal J_{q,T} \leq \mathcal{J}_q$, we have that $c_{q,T} \leq c_{q}$.\\

We show now that, for a suitable $\overline T>0$, we have that $\|u_{n}\|\leq \overline T$ and then,
being $K_{\overline T}(u_{n})=1$, $\{u_{n}\}$
is also a (PS) sequence for the untruncated functional $\mathcal J_{q}$,
at least for small values of $q$.

\begin{Lem}\label{lem:Tbarra}
	There exists $\overline{T}>0$ independent on $q$ and $q_{*}:=q(\overline{T})>0$ such that if $q<q_{*}$, then
	\[
	\limsup_{n} \|u_n\| \leq \overline{T}.
	\]
\end{Lem}
\begin{proof}
	Assume by contradiction that
	\begin{equation}
	\label{limsupT}
	\limsup_{n} \|u_n\| >T.
	\end{equation}
	Since
	\begin{align*}
	p \mathcal J_{q,T}(u_n) 
	- \mathcal J'_{q,T}(u_n)[u_n]
	&=
	\left(\frac{p}{2}-1\right)\|u_n\|^{2} 
	+\left(\frac{p}{4}-1\right) q^2K_T(u_n)\iint \frac{1-e^{-\frac{|x -y|}{a}}}{|x -y|} u_n^2(x)u_n^{2}(y)\\
	&\qquad
	- \frac{q^2}{2T^2} \chi'\left(\frac{\|u_n\|^2}{T^2}\right) \|u_n\|^2 \iint \frac{1-e^{-\frac{|x -y|}{a}}}{|x -y|} u_n^2(x)u_n^{2}(y)
	\end{align*}
	and so
	\begin{equation}
	\label{partenza}
	\begin{split}
	\left(\frac{p}{2}-1\right)\| u_n\|^{2}
	- \|\mathcal J'_{q,T}(u_n)\| \|u_n\|
	&\leq
	\left(\frac{p}{2}-1\right)\|u_n\|^{2}
	+ \mathcal J'_{q,T}(u_n)[u_n]\\
	&=
	p \mathcal J_{q,T}(u_n)
	- \left(\frac{p}{4}-1\right) q^2K_T(u_n)\iint \frac{1-e^{-\frac{|x -y|}{a}}}{|x -y|} u_n^2(x)u_n^{2}(y)\\
	&\qquad
	+ \frac{q^2}{2T^2} \chi'\left(\frac{\|u_n\|^2}{T^2}\right) \|u_n\|^2 \iint \frac{1-e^{-\frac{|x -y|}{a}}}{|x -y|} u_n^2(x)u_n^{2}(y)\\
	&\le
	p \mathcal J_{q,T}(u_n)
	+ \left|\frac{p}{4}-1\right| q^2K_T(u_n)\iint \frac{1-e^{-\frac{|x -y|}{a}}}{|x -y|} u_n^2(x)u_n^{2}(y)\\
	&\qquad
	+ \frac{q^2}{2T^2} \left|\chi'\left(\frac{\|u_n\|^2}{T^2}\right)\right| \|u_n\|^2 \iint \frac{1-e^{-\frac{|x -y|}{a}}}{|x -y|} u_n^2(x)u_n^{2}(y).
	\end{split}
	\end{equation}
	Let $w\in H^1(\R^3)$ be as in (\ref{MPTG3}) of Lemma \ref{lemMPGT}. Since $\mathcal J_{q,T}(u_n) \to c_{q,T}$ as $n\to +\infty$, there exists $\nu\in \mathbb{N}$ such that for all $n\geq \nu$
	\begin{align*}
	\mathcal J_{q,T}(u_n)
	&\leq 2 c_{q,T}
	\leq 2 \max_{t\in[0,1]} \J_{q,T}(tw)\\
	&\leq
	2 \max_{t\in[0,1]} \left[\frac{t^2}{2} \|w\|^{2}
	- \frac{t^p}{p}\|w\|_{p}^{p}\right]
	+ \frac{q^2}{2}\max_{t\in[0,1]} \left[t^4 K_T(tw)\iint \frac{1-e^{-\frac{|x -y|}{a}}}{|x -y|} w^2(x) w^{2}(y)\right]\\
	&:=
	2I_1 + \frac{q^2}{2}I_2.
	\end{align*}
	Observe that $I_1>0$.\\
	If $t^2 \|w\|^2>2T^2$ then $I_2=0$ and if $t^2 \|w\|^2\leq 2T^2$, then
	\[
	I_2\leq \frac{4T^4}{\|w\|^4} \iint \frac{1-e^{-\frac{|x -y|}{a}}}{|x -y|} w^2(x) w^{2}(y).
	\] 
	Thus
	\begin{equation}
	\label{primo}
	\mathcal J_{q,T}(u_n)\leq C_1+q^2 C_2 T^4.
	\end{equation}
	Analogously, we get that
	\begin{equation}
	\label{secondo}
	K_T(u_n)\iint \frac{1-e^{-\frac{|x -y|}{a}}}{|x -y|} u_n^2(x)u_n^{2}(y)
	\leq C_3 \|u_n\|^4
	\leq 4 C_3 T^4
	\end{equation}
	and
	\begin{equation}
	\label{terzo}
	\left|\chi'\left(\frac{\|u_n\|^2}{T^2}\right)\right| \|u_n\|^2 \iint \frac{1-e^{-\frac{|x -y|}{a}}}{|x -y|} u_n^2(x)u_n^{2}(y)
	\leq
	C_4 T^6.
	\end{equation}
	Putting \eqref{primo}--\eqref{terzo} in \eqref{partenza} we have
	\[
	\left(\frac{p}{2}-1\right)\| u_n\|^{2}
	- \|\mathcal J'_{q,T}(u_n)\| \|u_n\|
	\leq
	C_5+ q^2 C_6 T^4.
	\]
	On the other hand, since $\|\mathcal J'_{q,T}(u_n)\|\to 0$ as $n\to +\infty$ and by \eqref{limsupT}, we have
	\[
	\left(\frac{p}{2}-1\right)\| u_n\|^{2}
	- \|\mathcal J'_{q,T}(u_n)\| \|u_n\|
	\geq
	C_7 T^2 - T
	\]
	and so
	\[
	C_7 T^2 - T \leq C_5 + q^2 C_6 T^4
	\]
	which gives a contradiction if $q=q(T)$ is sufficiently small and for large $T$.
\end{proof}

\begin{Rem}\label{rem:a}
Observe that the above proof shows that $\overline T$ and $q_{*}=q(\overline T)$ do not depend on $a>0$.
\end{Rem}

Hence, for every $q\in (0, q_{*})$,
we have a  bounded (PS) sequence $\{u_{n}\}$, which actually depends on $q$ and $a$,
for the functional $\mathcal J_{q,\overline T}$
at the level $c_{q,\overline T}$.

However since the bound is exactly $\overline T$,
which gives $\mathcal J_{q,\overline T}(u_{n}) = 
\mathcal J_{q}(u_{n})$ and $c_{q,\overline T} =c_{q}$, we have
\begin{equation}\label{eq:PS}
\mathcal J_{q}(u_{n})\to c_{q}>0\,, \quad \mathcal J_{q}'(u_{n})\to 0,
\text{ as }n\to+\infty.  
\end{equation}

Moreover we can assume that $u_{n}\rightharpoonup u_0$ in $H^{1}(\mathbb R^{3})$.

The next result helps us to recover the compactness of the
bounded (PS) sequence $\{u_{n}\}$ we have found.
For the reader convenience, we give its proof in Appendix \ref{appsplitt}.

\begin{Lem}[Splitting]	\label{lemSplitting}
	Let $\{u_{n}\}$ be a bounded (PS) sequence for $\mathcal J_{q}$ at level $d>0$ and assume that $u_{n}\rightharpoonup u_{0}$ in $H^{1}(\mathbb R^{3})$.
	Then, up to subsequences, eihter $u_{n}$ strongly converges to $u_{0}$, or there exist 
		$\ell\in\mathbb{N}$, $\{z_n^{(k)}\}\subset \R^3$ for $1\leq k\leq\ell$, $w_1,\ldots,w_\ell\in H^1(\R^3)$ 
		such that
		\begin{enumerate}[label=(\roman*),ref=\roman*]
			\item \label{splitt2} $|z_n^{(k)}|\to +\infty$ for all $1\leq k\leq\ell$ and $|z_n^{(k)}-z_n^{(h)}|\to +\infty$ for all $1\leq k\neq h\leq\ell$, as $n\to + \infty$;
			\item \label{splitt3} $w_k\neq 0$ and $\mathcal{J}_q'(w_k)=0$ for all $1\leq k\leq\ell$;
			\item \label{splitt4} $\Big\| u_n - u_0 - \sum_{k=1}^\ell w_k(\cdot+z_n^{(k)})\Big\| = o_{n}(1)$;
			\item \label{splitt5} $d=\mathcal{J}_q(u_0)+ \sum_{k=1}^\ell \mathcal{J}_q(w_k)$;
			\item \label{splitt6} $\mathcal{J}_q(u_n)=\mathcal{J}_q(u_0)+ \sum_{k=1}^\ell \mathcal{J}_q(w_k)+o_{n}(1)$.
		\end{enumerate}
\end{Lem}

Then we can easily conclude the proof of Theorem \ref{th:existence}.\\
Indeed let $\{u_{n}\}$ the bounded (PS) sequence for $\mathcal J_{q}$ at level $c_{q}>0$
obtained in \eqref{eq:PS}. By Lemma \ref{lemSplitting} we have the following possibilities:
\begin{itemize}
\item if $u_n \to u_0$ we have finished being $u_0$ a solution;
\item if there exists $w\in\{u_0,w_1,\ldots,w_\ell\}$ such that $\mathcal{J}_q(w)<0$ we have finished, being $w$ a nontrivial solution;
\item if $\mathcal{J}_q(u_0), \mathcal{J}_q(w_1),\ldots,\mathcal{J}_q(w_\ell)\geq 0$, by (\ref{splitt5}) in Lemma \ref{lemSplitting} we have that
\[
c_{q}= \mathcal{J}_q(u_0)+ \sum_{k=1}^\ell \mathcal{J}_q(w_k)>0
\]
and we conclude.
\end{itemize}
From now on we will denote a generic solution by $\mathfrak{u}$.


\section{Existence for all $q$'s in the case $p\in(3,6)$}\label{sec:allq}
In this section we prove the existence of solutions of  \eqref{eq:ScBP} for every $q\neq 0$, but only for $p$ {\em large}, and, as we said in the Introduction, for radial symmetric functions.  Observe that if $u$ is radial, also $\phi_{u}$ is.

First we give some convergence properties in the radial setting recalling that $H_r^1(\mathbb{R}^3)$ is compactly embedded in $L^s(\R^3)$ for $s\in(2,6)$ by the celebrated Strauss Lemma \cite{strauss}.
\begin{Lem}\label{lemphi}
	If $u_n \rightharpoonup u$ in $H_r^1(\mathbb{R}^3)$, then 
	\begin{enumerate}[label=(\roman*),ref=\roman*]
		\item \label{lemphi2}$\phi_{u_{n}} \to \phi_u$ in $\mathcal{D}$;
		\item \label{lemphi3}$\displaystyle\int \phi_{u_{n}} u_{n}^{2} \to \int \phi_{u} u^{2}$.
	\end{enumerate}
\end{Lem}
\begin{proof}
	To prove (\ref{lemphi2}) we define the linear and continuous operators on $\mathcal{D}$
	$${\rm T}_{n}(\varphi)=\int \varphi u_{n}^{2} \quad \textrm{ and } \quad {\rm T}(\varphi)=\int \varphi u^{2}$$
	represented, by the Riesz Theorem, respectively by $\phi_{u_{n}} $ and $\phi_u$.
	Then, by the H\"older inequality,
	$$
	\|\phi_{u_{n}} - \phi_u\|_\mathcal{D}=
	\|{\rm T}_{n} - {\rm T}\|_{\mathcal{D}'}
	\leq C \|u_{n}^{2} - u^{2}\|_{6/5}
	\to 0.$$
	Moreover from (\ref{lemphi2}) and the H\"older inequality we easily get (\ref{lemphi3}).
\end{proof}

\subsection{The case $p\in[4,6)$.}
In this case any (PS) sequence for $\mathcal J_q$ is bounded. In fact, if $\{u_{n}\}\subset H^{1}(\mathbb R^{3})$ is a (PS) sequence, that is
	$$|\J_q(u_{n})|\leq M, \quad \J_q'(u_{n})=o_n(1),$$
	then, 
		\[
		pM + c\|u_{n}\|\geq p\J_q(u_{n}) - \J_q'(u_{n})[u_{n}]
		=\frac{p-2}{2}\|u\|^{2}+q^2\frac{p-4}{4}\int \phi_{u}u^{2}
		\geq C\|u_{n}\|^{2},
		\]
	from which the boundedness of $\{u_{n}\}$ follows.

The next Lemma is standard, since we have compactness.
\begin{Lem}\label{prop:PS}
	Let $p\in(2,6)$. Any bounded sequence $\{u_{n}\}\subset H^{1}_{r}(\mathbb R^{3})$
	such that $\mathcal \J_{q}'(u_{n})\to 0$
	has a convergent subsequence.
\end{Lem}
\begin{proof}
	We can assume, up to subsequence,
	that $u_{n}\rightharpoonup u$ in $H^{1}_{r}(\mathbb R^{3})$ and $u_{n}\to u$ in $L^{p}(\mathbb R^{3})$
	for $p\in (2,6)$. By Lemma \ref{lemphi} we have also that $\phi_{u_{n}}\to \phi_{u}$ in $\mathcal{D}$.
	By defining the Riesz  isomorphism 
	$\textrm{R}=-\Delta+\omega\textrm{I}:H^{1}_{r}(\mathbb R^{3})\to H^{-1}_{r}(\mathbb R^{3})$,
	by 
	$$\textrm R(u_{n}) +q^2 \phi_{u_{n}} u_{n}= |u_{n}|^{p-2}u_{n} +o_n(1) $$
	we have
	$$
	u_{n} = -q^2 \textrm{R}^{-1}(\phi_{u_{n}}u_{n} )+\textrm{R}^{-1}(|u_{n}|^{p-2}u_{n})+o_n(1).
	$$
	It is standard to see that each term in the right hand side is convergent in $H^{1}_{r}(\mathbb R^{3})$,
	however, for the reader's convenience, we give a short proof.
	Observe that
	$$\|\phi_{u_{n}} u_{n}\|_{3/2}\leq \|u_{n}\|_{2}\|\phi_{u_{n}}\|_{6}\leq \|u_{n}\| \|\phi_{u_{n}}\|_{\mathcal{D}}\leq C$$
	and, since by duality, $L_{r}^{3/2}(\mathbb R^{3})\hookrightarrow\hookrightarrow H_{r}^{-1}(\mathbb R^{3})$ we deduce
	$\{\phi_{u_{n}}u_{n}\}$ is convergent in $H_{r}^{-1}(\mathbb R^{3})$, and consequently 
	$\textrm{R}^{-1}(\phi_{u_{n}}u_{n} )$ also is.
	Analogously,
	\[
	\||u_{n}|^{p-2}u\|_{p'} = \|u_{n}\|^{p-1}_{p}\leq C \| u_{n}\|^{p-1} \leq C
	\]
	and again the compact embedding of  $L_{r}^{p'}(\mathbb R^{3})$ into $H_{r}^{-1}(\mathbb R^{3})$
	guarantees that  $\{|u_{n}|^{p-2}u_{n}\}$ is convergent into $H_{r}^{-1}(\mathbb R^{3})$ and so we conclude.
\end{proof}

Putting together Lemma \ref{Lem:MP}, Lemma \ref{prop:PS} and the boundedness of the (PS) sequences, the Mountain Pass Theorem allows to get a solution of \eqref{eq:ScBP}. 


\subsection{The case $p\in(3,4)$}
To study this case, we apply  the following result
\begin{Th}[{\cite[Theorem 1.1]{JJ}}]\label{th:Jj}
	Let $X$ be a Banach space with a norm $\|\cdot\|$ and let $L\subset \mathbb R^{+}$ be an interval.
	We consider a family $\{I_{\lambda}\}$ of $C^{1}$ functionals on $X$ of the form
	$$I_{\lambda}(u)= A(u) -\lambda B(u), \quad \forall \lambda\in L,$$
	where $B(u)\geq0$ for all $u\in X$ and either $A(u)\to +\infty$ or $B(u)\to+\infty$ as $\|u\|\to +\infty$. We assume that there exist two functions $v_{1},v_{2}\in X$ such that
	$$c_{\lambda}=\inf_{\gamma\in\Gamma} \max_{t\in[0,1]}I_{\lambda}(\gamma(t))>\max\{I_{\lambda}(v_{1}), I_{\lambda}(v_{2})\},\quad \forall \lambda\in L,$$
	$\Gamma = \{\gamma\in C([0,1], X) : \gamma(0) = v_{1}, \gamma(1)=v_{2}\}$. Then for almost all $\lambda\in L$,
	there exists a bounded (PS) sequence $\{u_{n}(\lambda)\}\subset X$ of $I_{\lambda}$
	at level $c_{\lambda}$.
\end{Th}

Let then,  for $\lambda\in [1/2,1]$,
\[
\J_{q,\lambda} (u)
:=
\frac{1}{2} \|\nabla u\|_2^2 + \frac{\omega}{2}  \|u\|_2^2 + \frac{q^2}{4} \int \phi_u u^2
-\frac{\lambda}{p} \|u\|_p^p.
\]
The Mountain Pass Geometry for $\J_{q,\lambda}$, which can be proved arguing as in Lemma \ref{Lem:MP}, ensures that 
$$c_{\lambda}:= \inf_{\gamma\in\Gamma_{q,\lambda}} \max_{t\in[0,1]} \J_{q,\lambda}(\gamma(t))>0,$$
where $\Gamma_{q,\lambda}=\{\gamma\in C([0,1]; H^{1}_{r}(\R^3)): \gamma(0)= 0, \gamma(1)= w\}$ and $w\in H_r^1(\R^3)$ is such that $\J_{q,\lambda}(w)<0$. Then Theorem \ref{th:Jj} gives, for a sequence $\{\lambda_{j}\}\subset [1/2,1]$ such that $\lim_{j}\lambda_{j}=1$,
a bounded (PS) sequence $\{u_{n, \lambda_{j}}\}$ at level $c_{\lambda_{j}}$ for the functional $\J_{q,\lambda_{j}}$. \\
Observe that, for all $j\in \mathbb N$, $c_{\lambda_{j}}\in[c_{1}, c_{1/2}]$.\\
In view of Lemma
\ref{prop:PS} we can assume that, up to subsequence, for every $j\in \mathbb N,\{u_{n,\lambda_{j}}\}$ 
strongly converges to some
$u_{\lambda_{j}}\in H^{1}_{r}(\R^3)$ satisfying
\begin{equation*}
\J_{q,\lambda_{j}}(u_{\lambda_{j}}) = c_{\lambda_{j}}, \quad \J_{q,\lambda_{j}}'(u_{\lambda_{j}})=0.
\end{equation*}
In particular, being such  $u_{\lambda_{j}}$ a solution of the equation
	\[
	- \Delta u +\omega u+q^{2}\phi_{u} u = \lambda_j |u|^{p-2}u 
	\qquad
	\mbox{in } 
	\mathbb{R}^{3},
	\]
	and, arguing as in Appendix \ref{PohozaevApp},
	it satisfies the Pohozaev identity
	\[
	\frac{1}{2} \|\nabla u_{\lambda_{j}}\|_2^2
	+\frac{3}{2}\omega \|u_{\lambda_{j}}\|_2^2
	-\frac{q^2}{16 \pi} \|\nabla\phi_j \|_2^2
	+ \frac{q^2 a^2}{16 \pi} \|\Delta\phi_j\|_2^2
	+\frac{3}{2}q^2\int \phi_j u_{\lambda_j}^2
	-\frac{3\lambda_{j}}{p} \|u_{\lambda_{j}}\|_p^p=0
	\]
	with $\phi_j:=\phi_{u_{\lambda_j}}$, which can be written also as
	\begin{equation}\label{eq:PohJ}
	-\|\nabla u_{\lambda_{j}}\|_2^2
	-\frac{q^2}{16 \pi} \|\nabla\phi_j \|_2^2
	+ \frac{q^2 a^2}{16 \pi} \|\Delta\phi_j\|_2^2
	+\frac{3}{4}q^2\int \phi_j u_{\lambda_j}^2
	+3c_{\lambda_{j}}
	=0.
	\end{equation}
	Moreover
	\begin{equation}
	\label{PohJbis}
	\begin{split}
	pc_{\lambda_{j}}
	&=
	p\J_{q,\lambda_{j}}(u_{\lambda_{j}}) - \J_{q,\lambda_{j}}'(u_{\lambda_{j}})[u_{\lambda_{j}}]\\
	&=
	\left(\frac{p}{2}-1\right)\|\nabla u_{\lambda_{j}}\|_2^2
	+ \left(\frac{p}{2}-1\right)\omega \| u_{\lambda_{j}}\|_2^2
	+q^2 \left(\frac{p}{4}-1\right)\int \phi_j u_{\lambda_j}^2.
	\end{split}
	\end{equation}
	Thus, using \eqref{eq:PohJ} and \eqref{PohJbis}, we get
	\begin{equation}\label{eq:paperino}
	\begin{split}
	(p-3)\|\n u_{\lambda_{j}}\|_2^2 +\frac{p-2}{2}\omega\|u_{\lambda_{j}}\|_2^2
	&=
	\left(\frac{p}{2}-2\right)\|\nabla u_{\lambda_{j}}\|_2^2
	+\left(\frac{p}{2}-1\right)\|\nabla u_{\lambda_{j}}\|_2^2
	+ \left(\frac{p}{2}-1\right)\omega \| u_{\lambda_{j}}\|_2^2 
	\\
	&=
	\left(\frac{5}{2}p-6\right)c_{\lambda_{j}}
	+\frac{q^2}{8\pi} \left(\frac{p}{4}-1\right) \left(	-\|\nabla\phi_j \|_2^2
	+ a^2 \|\Delta\phi_j\|_2^2
	+4\pi \int \phi_j u_{\lambda_j}^2\right)\\
	&=
	\left(\frac{5}{2}p-6\right)c_{\lambda_{j}}
	+\frac{q^2 a^2}{4\pi} \left(\frac{p}{4}-1\right) \|\Delta\phi_j\|_2^2  \\
	&\leq \frac{5p-12}{2}c_{1/2}
	\end{split}
	\end{equation}
	which gives directly the boundedness of $\{u_{\lambda_{j}}\}$.\\
	We show now that $\{u_{\lambda_{j}}\}$ is indeed a (PS) sequence for the unperturbed functional
	$\J_{q}$. Indeed  due to the boundedness of $\{u_{\lambda_{j}}\}$ in $H^{1}_{r}(\R^3)$:
	$$
	\J_{q}(u_{\lambda_{j}}) =\J_{q,\lambda_{j}}(u_{\lambda_{j}}) -(1-\lambda_{j})\|u_{\lambda_{j}}\|_{p}^{p}\\
	= c_{\lambda_{j}} + o_{j}(1)
	$$
	implying the boundedness of $\{\J_{q}(u_{\lambda_{j}})\}$. Moreover 
	\begin{align*}
	\sup_{\|v\|\leq 1} |\J_{q}'(u_{\lambda_{j}})[v]| 
	&=
	\sup_{\|v\|\leq 1} \Big|\J_{q,\lambda_{j}}'(u_{\lambda_{j}})[v]
	-(1-\lambda_{j})p \int |u_{\lambda_{j}}|^{p-2}u_{\lambda_{j}} v \Big| \\
	&\leq
	\| \J'_{q,\lambda_{j}}(u_{\lambda_{j}})\| +(1-\lambda_{j})p C \|u_{\lambda_{j}}\|_{p}^{p-1}\\
	&= o_{j}(1).
	\end{align*}
	But then, in view of  Lemma \ref{prop:PS}, up to subsequence
	$u_{\lambda_{j}} \to \mathfrak{u}$ and so $\J_q'(\mathfrak{u})=0$, meaning that $\mathfrak{u}$ is a solution of \eqref{eq:ScBP} we were looking for. 

We conclude the section with the following remark that will be useful in the next section.


\begin{Rem}\label{rem:Tbound}
In the radial setting we can repeat the arguments in Section \ref{sec4} replacing the Splitting Lemma \ref{lemSplitting} by standard arguments using the compact embedding of $H_r^1(\R^3)$ into $L^p(\R^3)$, $p\in(2,6)$, Lemma \ref{lem:wlimit}, and Lemma \ref{lemphi}. In such a way, using the notations introduced in Lemma \ref{lem:Tbarra}, for every $q<q_{*}$ we get a solution $\mathfrak{u}$ such that $\|\mathfrak{u}\|\leq \overline{T}$.
\end{Rem}

%
%
%
%
%
%
%
%


\section{The behaviour as $a\to 0$ in the radial case}\label{sec:Limite}

This section is devoted to the proof of Theorem \ref{th:ato0}. We begin by showing the following preliminary result. 
\begin{Lem}\label{lem:ato0}
 Consider $f^0\in  L^{6/5}(\mathbb R^{3})$,  $\{ f_{a}\}_{ a\in (0,1)}\subset  L^{6/5}(\mathbb R^{3})$ and let 
 $$\phi^{0}\in D^{1,2}(\mathbb R^{3}) \hbox{ be the unique solution of }  -\Delta \phi = f^{0} \hbox{ in }\mathbb R^{3}$$
 and
 $$\phi^{a}\in \mathcal D\hbox{ be the unique solution of } -\Delta\phi + a^{2}\Delta^{2}\phi = f^{a}\hbox{ in }\mathbb R^{3}.$$ 
As  $a\to0$ we have:
\begin{enumerate}[label=(\roman*),ref=\roman*]
		\item\label{1ato0} if $f^{a}\rightharpoonup f^{0}$ in $L^{6/5}(\mathbb R^{3})$, then 
		$\phi^{a}\rightharpoonup \phi^{0}$ in $D^{1,2}(\mathbb R^{3})$;
		\item\label{2ato0}  if $f^{a}\to f^{0}$ in $L^{6/5}(\mathbb R^{3})$, then $\phi^{a}\to \phi^{0}$ in $D^{1,2}(\mathbb R^{3})$ and $a\Delta \phi^{a} \to 0$ in $L^{2}(\mathbb R^{3})$.
	\end{enumerate}
\end{Lem}
\begin{proof}
	By
	$$
	\|\nabla \phi^{a}\|_2^{2} +a^{2} \|\Delta \phi^{a}\|_2^{2} = \int f^{a}\phi^{a}
	\leq C \|f^{a}\|_{6/5} \|\nabla \phi^{a}\|_{2}
	$$
	we deduce
	\begin{equation*}
	\|\nabla \phi^{a}\|_{2} \leq C \|f^{a}\|_{6/5}, \quad \|a\Delta \phi^{a}\|_{2} \leq C\|f^{a}\|_{6/5}.
	\end{equation*}
	Then there exists $\phi_{*}\in D^{1,2}(\mathbb R^{3})$ such that $\phi^{a}\rightharpoonup \phi_{*}$ in $D^{1,2}
	(\mathbb R^{3})$.
	Passing to the limit as $a\to 0$ in the identity
	\begin{equation*}
	\int \nabla \phi^{a} \nabla \varphi + a^{2}\int \Delta \phi^{a} \Delta \varphi =\int f^{a} \varphi, \quad \varphi\in C^{\infty}_{c}(\mathbb 	R^{3}),
	\end{equation*}
	and using that
	$$\Big|a^{2}\int \Delta \phi^{a} \Delta \varphi \Big| \leq a \|a\Delta \phi^{a}\|_{2} \|\Delta \varphi\|_{2} \leq a C\to 0,$$
	we get
	\begin{equation*}
	\int \nabla \phi_{*} \nabla \varphi = \int f^{0} \varphi.
	\end{equation*}
	Then, by the uniqueness, $\phi_{*} = \phi^{0}$ proving \eqref{1ato0}.\\	
	Assume now that $f^{a}\to f^0$ in $L^{6/5}(\mathbb R^{3})$. Of course we have
	\begin{equation}\label{eq:liminf}
	\|\nabla \phi^{0}\|_{2}^{2} \leq \liminf_{a\to 0} \|\nabla \phi^{a}\|_{2}^{2}.
	\end{equation}
	Let $\{\varphi_{n}\}\subset C_{c}^{\infty}(\mathbb R^{3})$ such that $\varphi_{n} \to \phi^{0} $ in $D^{1,2}(\mathbb R^{3})$ as $n\to+\infty$.
	Since $\phi^{a}$ minimizes the functional
	$$E_{a}(\phi) = \frac{1}{2}\|\nabla \phi\|_2^{2} + \frac{a^{2}}{2}\|\Delta \phi\|_2^{2} - \int  f^{a}\phi, 
	\quad \phi \in \mathcal D,$$
	we have
	\begin{align*}
	\frac{1}{2} \|\nabla \phi^{a}\|_{2}^{2}
	&=
	E_{a}(\phi^{a}) - \frac{a^{2}}{2} \|\Delta \phi^{a}\|_2^{2} +
	\int f^{a} \phi^{a} \\ 
	&\leq
	E_{a}(\varphi_{n}) +\int f^{a}\phi^{a} \\
	&=
	\frac{1}{2}\|\nabla \varphi_{n}\|_2^{2} +\frac{a^{2}}{2} \|\Delta \varphi_{n}\|_2^{2} - \int f^{a}\varphi_{n}  +\int f^{a}\phi^{a}
	\end{align*}
	and then
	\begin{equation}\label{eq:limsupa}
	\limsup_{a\to 0} \frac{1}{2} \|\nabla \phi^{a}\|_2^{2}\leq
	\frac{1}{2}\|\nabla \varphi_{n}\|_2^{2}  - \int f^{0} \varphi_{n}  +\int f^{0} \phi^{0}.
	\end{equation}
	Passing to the limit in $n$ in \eqref{eq:limsupa} we get
	\begin{equation}\label{eq:limsup}
	\limsup_{a\to 0}  \|\nabla \phi^{a}\|_2^{2}\leq
	\|\nabla \phi^{0}\|_2^{2}.
	\end{equation}
	By \eqref{eq:liminf}, \eqref{eq:limsup} and  the convergence
	$\phi^{a} \rightharpoonup \phi^{0}$ in $D^{1,2}(\mathbb R^{3})$ we infer $\phi^{a}\to \phi^{0}$
	in $D^{1,2}(\mathbb R^{3})$.\\
	Finally we see that, for $a\to 0$,
	$$\|a\Delta \phi^{a}\|_{2}^{2} = \int f^{a}\phi^{a} - \|\nabla \phi^{a}\|_2^{2} \to \int f^{0} \phi^{0}
	- \|\nabla \phi^{0}\|_2^{2}=0$$
	and the proof is complete.
\end{proof}


Now we are ready to prove Theorem \ref{th:ato0}.
\begin{proof}[Proof of Theorem \ref{th:ato0}]

From now on we fix $q$ according to the restriction of Theorem \ref{th:existence},
and let  $\{\mathfrak u^{a}, \phi^{a}\}\subset H_r^{1}(\mathbb R^{3})\times \mathcal D_r$ 
be the family of the solutions of \eqref{eq:ScBP} for this fixed value $q$.
We are using the notation $\phi^{a}:= \phi^{a}_{ \mathfrak{u}^{a}}$.
In contrast to the previous sections we use the explicit dependence on $a$ also in the functional.
Then, 
the functions $\{\mathfrak u^{a}\}$ solve
\begin{equation*}
- \Delta u +\omega u+ q^{2}\phi^{a} u = |u|^{p-2}u  \mbox{ in }
\mathbb{R}^{3}
\end{equation*}
and are critical point of $\mathcal J_{ q}^{a}$ at the Mountain Pass value $c_{ q}^{a}>0$.\\
Our first aim is to show the boundedness of $\{\mathfrak u^{a}\}_{a\in(0,1)}$ in $H_r^{1}(\mathbb R^{3})$.\\
Let
\[
c_{q}^0:=\inf_{\gamma\in\Gamma^{0}}\max_{t\in[0,1]} \mathcal J^{0}(\gamma(t)) \geq c_{ q}^{a}>0,
\]
where $
\Gamma^0
:=
\left\{\gamma \in C([0,1],H_r^1(\R^3)):  \gamma(0)=0, \mathcal J^0(\gamma(1))<0 \right\}$,
\[
\J^0 (u)
:=
\frac{1}{2} \|\nabla u\|_2^2 + \frac{\omega}{2}  \|u\|_2^2 + \frac{q^2}{4} \iint \frac{ u^2(x) u^2(y)}{|x-y|}
-\frac{1}{p} \|u\|_p^p
\]
is the functional related to \eqref{eq:SP}.\\
We distinguish three cases.\\
{\bf Case A:} $p\in[4,6)$.\\
We have
\[
c_{ q}^0 \geq c_{ q}^{a} = \mathcal J_{ q}^{a}(\mathfrak u^{a}) - 
\frac1 p  (\mathcal J_{ q}^{a})'(\mathfrak u^{a})[\mathfrak u^{a}]
=\frac{p-2}{2p} \|\mathfrak u^{a}\|^{2} +  \frac{p-4}{4p} \int \phi^{a} (\mathfrak u^{a})^{2}
\geq \frac{p-2}{2p} \|\mathfrak u^{a}\|^{2} .
\]
{\bf Case B:} $p\in(3,4)$.\\
Arguing as in \eqref{eq:paperino}, since $\mathfrak u^{a}$ is a solution at the Mountain Pass level
 $c_{ q}^{a}$, we infer
$$	
(p-3)\|\n \mathfrak u^{a}\|_2^2 +\frac{p-2}{2}\omega\|\mathfrak u^{a}\|_2^2
\leq \frac{5p-12}{2}c_{ q}^{a}\leq \frac{5p-12}{2}c_{ q}^0.
$$
{\bf Case C:} $p\in(2,3]$.\\
By Remark \ref{rem:Tbound} we already know that $\|\mathfrak u^{a}\|\leq \overline T$, with $\overline T$ that does not depend on $a$ (see Remark \ref{rem:a}).\\
Hence in any case $\{\mathfrak u^{a}\}$ is bounded in $H^{1}_{r}(\mathbb R^{3})$ and
there exists $\mathfrak u^{0}\in H_r^{1}(\mathbb R^{3})$ such that, up to subsequences,
$\mathfrak u^{a}\rightharpoonup \mathfrak u^{0}$  in $H^{1}_{r}(\mathbb R^{3})$ as $ a\to0$.
In particular $(\mathfrak u^{a})^{2} \to (\mathfrak u^{0})^{2}$ in $L^{6/5}(\mathbb R^{3})$ and
by \eqref{2ato0} of Lemma \ref{lem:ato0} we infer that  $\phi^{a} \to \phi^{0}$, where $\phi^{0}\in D_r^{1,2}(\mathbb R^{3})$ is the unique solution of $ -\Delta \phi =4\pi \mathfrak{u}_0^2$ in $\mathbb R^{3}$.
The fact that $\mathfrak u^{a}\to \mathfrak u^{0}$ in $H_r^{1}(\mathbb R^{3})$
 is done  
as in Lemma \ref{prop:PS} since the proof can be merely repeated using Lemma \ref{lem:ato0}.\\
Let now $\varphi\in C^{\infty}_{c}(\mathbb R^{3})$ with $\operatorname{supp}(\varphi)=\Omega$. We know that
$$\langle \mathfrak u^{a}, \varphi\rangle +q^{2} \int_{\Omega} \phi^{a} \mathfrak u^{a} \varphi = 
\int_{\Omega} |\mathfrak u^{a}|^{p-2} \mathfrak u^{a}\varphi.$$
We want to pass to the limit as $a\to 0$ in each term.
Of course
\begin{equation}\label{eq:prodscal}
\langle \mathfrak u^{a}, \varphi\rangle \to \langle \mathfrak u^{0}, \varphi\rangle,
\end{equation}
and, as follows by standard arguments,
\begin{equation*}
\int_{\Omega}|\mathfrak u^{a}|^{p-2} \mathfrak u^{a} \varphi \to
 \int_{\Omega}|\mathfrak u^{0}|^{p-2} \mathfrak u^{0} \varphi.
\end{equation*}
Moreover, since $\phi^{a} \to \phi^{0}$ in $L^{6}(\mathbb R^{3}), 
\mathfrak u^{a}\to \mathfrak u^{0}$ in $L^{12/5}(\Omega)$
and $\varphi\in L^{12/5}(\Omega)$, by the H\"older inequality we easily get
\begin{equation}\label{eq:nonlocalterm}
\int_{\Omega}\phi^{a} \mathfrak u^{a}\varphi \to  \int_{\Omega} \phi^{0} \mathfrak u^{0} \varphi.
\end{equation}
Then by \eqref{eq:prodscal}-\eqref{eq:nonlocalterm} we arrive at
\begin{equation*}
\langle \mathfrak u^{0}, \varphi \rangle + q^{2}\int_{\Omega} \phi^{0} \mathfrak u^{0} \varphi = \int_{\Omega} |\mathfrak u^{0}|^{p-2} \mathfrak u^{0}\varphi
\end{equation*}
which shows that $(\mathfrak u^{0},\phi^{0})$ solves \eqref{eq:SP}.
\end{proof}

\appendix

\section{Properties of solutions and nonexistence}\label{app:A}

In this appendix we show that our solutions are indeed classical. Moreover we prove, by means of Nehari and Poho\v zaev type identities, some nonexistence results.

\subsection{Regularity of the solutions.}\label{secreg}
We remark here that the weak solutions we find are indeed classical solutions.
This is based on standard bootstrap arguments that we briefly recall  here.\\
Let us first observe that if $(u,\phi)\in H^{1}(\mathbb R^{3})\times \mathcal D$
is a weak solution of \eqref{eq:ScBP} then $\psi: = -a^{2}\Delta \phi+\phi$
solves weakly, in any bounded domain $\Omega$, the equation
$$-\Delta \psi = 4\pi u^{2} \quad \text{ in }\Omega.$$
Now, being $u^{2}\in L^{3}(\mathbb R^{3})$ 
it holds
(see e.g. \cite[Theorem 9.9]{GT})
\begin{equation}\label{eq:1}
-a^{2}\Delta \phi +\phi= \psi \in W_{\textrm{loc}}^{2,3}(\mathbb R^{3}).
\end{equation}
Since $\phi\in H^{1}_{\textrm{loc}}(\mathbb R^{3})$ is a weak solution of \eqref{eq:1}
with $\psi\in W^{2,2}_{\textrm{loc}}(\mathbb R^{3})$,
by {\sl higher interior regularity} (see e.g. \cite[Theorem 8.10]{GT}), we deduce
$\phi\in W^{4,2}_{\textrm{loc}}(\mathbb R^{3})$ and by the Sobolev embedding 
(see e.g. \cite[Theorem 5.4]{Adams}) we deduce that
$\phi\in C^{2,\lambda}_{\textrm{loc}}(\mathbb R^{3})$, $\lambda\in (0,1/2]$.\\
%
%
Then considering the equation 
$$-\Delta u + \omega u+q^{2}\phi u =|u|^{p-2}u$$
we deduce by bootstrap arguments that $u\in C^{2,\lambda}_{\textrm{loc}}(\mathbb R^{3})$.
But then, being
\begin{equation*}
-\Delta \psi = 4\pi  u^{2}\in H^{2}_{\textrm{loc}}(\mathbb R^{3}), 
\end{equation*}
it holds again by \cite[Theorem 8.10]{GT} that 
$$-a^{2}\Delta \phi +\phi =\psi \in H^{4}_{\textrm{loc}}(\mathbb R^{3})$$
and then, by higher interior regularity and Sobolvev embedding, $\phi\in H^{6}_{\textrm{loc}}(\mathbb R^{3})\hookrightarrow C^{4,\lambda}_{\textrm{loc}}(\mathbb R^{3})$, $\lambda\in (0,1/2]$.

\subsection{A useful identity}\label{id}

We define the Fourier transform of a function $f\in L^{1}(\mathbb R^{3})$  by the formula
$$\mathcal F(f) (x):= \frac{1}{(2\pi)^{3/2}} \int e^{-i x y} f(y) dy.$$
If $f\in L^{2}(\mathbb R^{3})$ its Fourier transform is defined by the usual approximation procedure.
With this definition we have, for $f,g\in L^{2}(\mathbb R^{3})$,
$$\mathcal F(f*g)=(2\pi)^{3/2}\mathcal F(f) \mathcal F (g) \quad \text{and}
\quad \int \mathcal F(f) \mathcal F(g) = \int fg.$$

Then, by classical results in Fourier analysis (observe that the functions involved are 
e.g. in $L^{1}(\mathbb R^{3})\cap L^{2}(\mathbb R^{3})$), 
we get for any $a>0$,
$$
\mathcal F( e^{-|\cdot |/a}) (x)=\sqrt{\frac{2}{\pi}} \frac{2a^{3}}{(1+a^{2}|x|^{2})^{2}}\,
\qquad  \text{and}
\qquad \mathcal F\Big(\frac{e^{-|\cdot|/a}}{|\cdot|}\Big)(x)=\sqrt{\frac{2}{\pi}}\frac{a^{2}}{1+a^{2}|x|^{2}}.
$$
Then by recalling \eqref{eq:DeltaK} we have 
\begin{align}\label{eq:L*L}
\mathcal F(\Delta \mathcal K *\Delta \mathcal K)= (2\pi)^{3/2} \mathcal F(\Delta \mathcal K) \mathcal F(\Delta \mathcal K)=
 \frac{4\sqrt{2\pi}}{(1+a^{2}|\cdot |^{2})^{2}}
\end{align}
from which we deduce
$$
\Delta \mathcal K * \Delta\mathcal K = \frac{2\pi}{a^{3}} e^{-|\cdot |/a}.
$$
Since for $u\in H^{1}(\mathbb R^{3})$ it is $\Delta \phi_{u} = \Delta \mathcal K*u^{2}$ (see Lemma \ref{lemfundsol} and item \eqref{propphivi}
in Lemma \ref{lem:propphi}),  in virtue of \eqref{eq:L*L}, we get
\begin{equation}
\label{Fourier}
\begin{split}
\|\Delta \phi_u\|_2^2 
&= \int |\mathcal F (\Delta \phi_{u})|^2 \\
&=(2\pi)^{3/2}\int  \mathcal F(u^{2})\, \mathcal F(\Delta \mathcal K) \mathcal F (\Delta \phi_{u}) \\
& = \int  \mathcal F (u^{2})\, \mathcal F (\Delta \mathcal K * \Delta\phi_{u}) \\ 
& = \int   u^{2}\,  \Delta \mathcal K * \Delta \mathcal K *u^{2} \\ 
& = \frac{2\pi}{a^{3}}\iint e^{-|x-y|/a} u^{2}(x)u^{2}(y)dxdy.
\end{split}
\end{equation}
In another words we have the identity
\begin{equation*}
\int \Big(\int \frac{e^{-|x-y|/a}}{|x-y|}u^{2}(y) dy \Big)^{2}dx =  2\pi a\iint e^{-|x-y|/a} u^{2}(x)u^{2}(y)dxdy
\end{equation*}
true for any $u\in H^{1}(\mathbb R^{3})$.

\subsection{The Poho\v zaev identity}\label{PohozaevApp}
Let $(u,\phi)\in H^{1}(\mathbb R^{3})\times \mathcal D$ be a nontrivial solution of \eqref{eq:ScBP}. Recall that $\phi=\phi_u$. We have
\begin{equation}
\label{Ne1}
\|\nabla u\|_2^2
+ \omega \|u\|_2^2
+ q^2 \int \phi u^2
- \|u\|_p^p =0
\end{equation}
and
\begin{equation}
\label{Ne2}
\|\nabla \phi\|_2^2 + a^2 \|\Delta \phi\|_2^2= 4\pi \int  \phi u^2,
\end{equation}
that are usually called {\em Nehari} identities.\\
Moreover $(u,\phi)$ satisfies also the following {\em Poho\v zaev} identity
\begin{equation}
\label{poho}
- \frac{1}{2}\|\nabla u\|_2^2
-\frac{3}{2} \omega \| u\|_2^2
+ \frac{q^2}{16\pi} \|\nabla \phi\|_2^2
- \frac{q^2 a^2}{16\pi} \|\Delta\phi\|_2^2
- \frac{3}{2}q^2\int \phi u^2
+\frac{3}{p}\|u\|_p^p
=0.
\end{equation}
In fact, if $(u,\phi)$ solve \eqref{eq:ScBP}, recalling the regularity proved in Section \ref{secreg}, for every $R>0$, we have
\begin{align}
\int_{B_R} -\Delta u (x\cdot\nabla u)
&=
- \frac{1}{2}\int_{B_R}|\nabla u|^2
-\frac{1}{R}\int_{\partial B_R} |x\cdot\nabla u|^2
+ \frac{R}{2}\int_{\partial B_R} |\nabla u|^2, \label{eq:pohoBR1}
\\
\int_{B_R} \phi u (x\cdot\nabla u)
&=
- \frac{1}{2} \int_{B_R} u^2 (x\cdot \nabla\phi)
- \frac{3}{2}\int_{B_R}\phi u^2
+\frac{R}{2}\int_{\partial B_R} \phi u^2, \label{eq:pohoBR2}
\\
\int_{B_R} u (x\cdot\nabla u)&=-\frac{3}{2} \int_{B_R} u^2 + \frac{R}{2}\int_{\partial B_R} u^2, \label{eq:pohoBR3}
\\
\int_{B_R} |u|^{p-2} u  (x\cdot\nabla u)&=-\frac{3}{p}\int_{B_R} |u|^p + \frac{R}{p}\int_{\partial B_R} |u|^p, \label{eq:pohoBR4}
\end{align}
where $B_R$ is the ball of $\RT$ centered in the origin and with radius $R$ (see also \cite{DM}), and, since
\[
\Delta^2 \phi (x\cdot \nabla\phi)
= \operatorname{div}\left(
\nabla\Delta \phi (x\cdot\nabla\phi)
- \Delta\phi\nabla\phi
- \mathbb{F}
+x\frac{(\Delta\phi)^2}{2}\right)
+ \frac{(\Delta\phi)^2}{2},
\]
where $\mathbb{F}_i = \Delta\phi (x\cdot \nabla(\partial_i \phi))$, $i=1,2,3$, then
\begin{equation}
\label{eq:pohoBR5}
\int_{B_R} \Delta^2 \phi (x\cdot \nabla\phi)
=
\frac{1}{2}\int_{B_R} (\Delta\phi)^2
+\int_{\partial B_R} 
\left(
\nabla\Delta \phi (x\cdot\nabla\phi)
- \Delta\phi\nabla\phi
- \mathbb{F}
+x\frac{(\Delta\phi)^2}{2}\right)\cdot\nu.
\end{equation}
Multiplying the first equation of \eqref{eq:ScBP} by $x\cdot\nabla u$ and the second equation
by $x\cdot\nabla\phi$ and integrating on $B_R$, by \eqref{eq:pohoBR1}, \eqref{eq:pohoBR2}, \eqref{eq:pohoBR3}, \eqref{eq:pohoBR4}, and \eqref{eq:pohoBR5} we get
\begin{equation}
\label{eq:pohotBR1}
\begin{split}
- \frac{1}{2}&\int_{B_R}|\nabla u|^2
-\frac{1}{R}\int_{\partial B_R} |x\cdot\nabla u|^2
+ \frac{R}{2}\int_{\partial B_R} |\nabla u|^2
-\frac{3}{2} \omega \int_{B_R} u^2
+ \frac{R}{2}\omega \int_{\partial B_R} u^2\\
&
- \frac{q^2}{2} \int_{B_R} u^2 (x\cdot \nabla\phi)
- \frac{3}{2}q^2\int_{B_R}\phi u^2
+q^2\frac{R}{2}\int_{\partial B_R} \phi u^2
=-\frac{3}{p}\int_{B_R} |u|^p + \frac{R}{p}\int_{\partial B_R} |u|^p
\end{split}
\end{equation}
and
\begin{equation}\label{eq:pohotBR2}
\begin{split}
4\pi \int_{B_R} u^2 (x\cdot\nabla\phi)
&=- \frac{1}{2}\int_{B_R}|\nabla \phi|^2
- \frac{1}{R}\int_{\partial B_R} |x\cdot\nabla \phi|^2 
+ \frac{R}{2}\int_{\partial B_R} |\nabla \phi|^2
+\frac{a^2}{2} \int_{B_R} (\Delta\phi)^2\\
&\qquad
+ a^2 \int_{\partial B_R} 
\left(
\nabla\Delta \phi (x\cdot\nabla\phi)
- \Delta\phi\nabla\phi
- \mathbb{F}
+x\frac{(\Delta\phi)^2}{2}\right)\cdot\nu.
\end{split}
\end{equation}
Substituting \eqref{eq:pohotBR2} into \eqref{eq:pohotBR1} we obtain
\begin{align*}
- \frac{1}{2}&\int_{B_R}|\nabla u|^2
-\frac{3}{2} \omega \int_{B_R} u^2
+ \frac{q^2}{16\pi} \int_{B_R}|\nabla \phi|^2
- \frac{q^2 a^2}{16\pi} \int_{B_R} (\Delta\phi)^2
- \frac{3}{2}q^2\int_{B_R}\phi u^2
+\frac{3}{p}\int_{B_R} |u|^p 
\\
&
=
\frac{1}{R}\int_{\partial B_R} |x\cdot\nabla u|^2
- \frac{R}{2}\int_{\partial B_R} |\nabla u|^2
- \frac{R}{2}\omega \int_{\partial B_R} u^2
-q^2\frac{R}{2}\int_{\partial B_R} \phi u^2
+ \frac{R}{p}\int_{\partial B_R} |u|^p\\
&\qquad
- \frac{q^2}{8\pi R}\int_{\partial B_R} |x\cdot\nabla \phi|^2 
+ \frac{q^2 R}{16\pi}\int_{\partial B_R} |\nabla \phi|^2\\
&\qquad
+ \frac{q^2a^2}{8\pi}  \int_{\partial B_R} 
\left(
\nabla\Delta \phi (x\cdot\nabla\phi)
- \Delta\phi\nabla\phi
- \mathbb{F}
+x\frac{(\Delta\phi)^2}{2}\right)\cdot\nu.
\end{align*}
Using the same arguments as in \cite[Proof of Theorem 1.1]{DM} we have that the right hand side tends to zero as $R\to +\infty$, since
\begin{align*}
&\int_{\partial B_R} \nabla\Delta \phi (x\cdot\nabla\phi)\cdot\nu
= R \int_{\partial B_R} \frac{\partial \Delta \phi}{\partial \nu}  \frac{\partial \phi}{\partial \nu} \to 0,
\\
&\int_{\partial B_R} \Delta\phi\nabla\phi\cdot\nu
= \int_{\partial B_R} \Delta\phi \frac{\partial \phi}{\partial \nu} \to 0,
\\
&\int_{\partial B_R} 
\mathbb{F} \cdot\nu
=R  \int_{\partial B_R} \frac{\partial^2 \phi}{\partial \nu^2} \to 0,
\\
&\frac{1}{2}\int_{\partial B_R} (\Delta\phi)^2 x\cdot\nu = \frac{R}{2}\int_{\partial B_R} (\Delta\phi)^2  \to 0,
\end{align*}
and so we get \eqref{poho}.\\
Finally we observe that, using \eqref{Fourier}, the Poho\v zaev identity \eqref{poho} can be written also as
\begin{equation}
	\label{eq:pohoapp}
	-\frac{1}{2}\|\nabla u\|_2^2
	-\frac{3}{2} \omega \|u\|_2^2
	- \frac{q^2}{4a}\iint \left[5 \frac{1-e^{-\frac{|x-y|}{a}}}{|x-y|/a} + e^{-\frac{|x-y|}{a}}\right] u^2(x) u^2(y) dxdy
	+\frac{3}{p}\|u\|_p^p
	=0.
	\end{equation}

\subsection{A nonexistence result}

Using the identities recalled before, we are able to show nonexistence results for $p\leq 2$ and for $p\geq 6$.\\
In fact,
if $(u,\phi)\in H^{1}(\mathbb R^{3})\times \mathcal D$ is a nontrivial solution of \eqref{eq:ScBP} and $p\geq 6$, replacing \eqref{Ne1} and \eqref{Ne2} into \eqref{poho} we get
\begin{align*}
0
&=
- \frac{1}{2}\|\nabla u\|_2^2
-\frac{3}{2} \omega \| u\|_2^2
+ \frac{q^2}{16\pi} \|\nabla \phi\|_2^2
- \frac{q^2 a^2}{16\pi} \|\Delta\phi\|_2^2
- \frac{3}{2}q^2\int \phi u^2
+\frac{3}{p}\|u\|_p^p\\
&=
\left(\frac{3}{p}-\frac{1}{2}\right)\|\nabla u\|_2^2
+\left(\frac{3}{p} - \frac{3}{2}\right)\omega \|u\|_2^2
- \frac{q^2 a^2}{8\pi} \|\Delta\phi\|_2^2
+\left(\frac{3}{p} - \frac{5}{4}\right)q^2\int \phi u^2\\
&\leq
-\omega \|u\|_2^2<0.
\end{align*}
Moreover, if $p\leq 2$, replacing \eqref{eq:pohoapp} into \eqref{Ne2} and using \eqref{phiu}, we have
	\begin{align*}
	0
	&=
	\|\nabla u\|_2^2
	+ \omega \|u\|_2^2
	+ q^2\int \phi u^2
	- \|u\|_p^p\\
	&=
	\left(1-\frac{p}{6}\right)\|\nabla u\|_2^2
	+ \left(1-\frac{p}{2}\right) \omega \|u\|_2^2
	+ \left(1-\frac{5}{12}p\right) q^2\iint \frac{1-e^{-\frac{|x -y|}{a}}}{|x -y|}u^2(x)u^2(y)dxdy\\
	&\qquad
	-q^2 \frac{p}{12 a} \iint e^{-\frac{|x -y|}{a}} u^2(x)u^2(y)dxdy\\
	&=
	\left(1-\frac{p}{6}\right)\|\nabla u\|_2^2
	+ \left(1-\frac{p}{2}\right) \omega \|u\|_2^2\\
	&\qquad
	+\frac{q^2}{a}  
	\iint \left[
	\left(1-\frac{5}{12}p\right) \frac{1-e^{-\frac{|x -y|}{a}}}{|x -y|/a}
	-\frac{p}{12}e^{-\frac{|x -y|}{a}} 
	\right]u^2(x)u^2(y)dxdy\\
	&\geq
	\frac{2}{3}\|\nabla u\|_2^2
	+\frac{q^2}{6a}  
	\iint \left[
	\frac{1-e^{-\frac{|x -y|}{a}}}{|x -y|/a}
	-e^{-\frac{|x -y|}{a}} 
	\right]u^2(x)u^2(y)dxdy >0
	\end{align*}
since the function in the parenthesis is positive.

\section{Proof of Lemma \ref{lemSplitting}}\label{appsplitt}
This appendix is devoted to the proof of the Splitting Lemma. To do this, we need some preliminary results.
\begin{Lem}\label{lem:wlimit}
	The weak limit of a (PS) sequence for $\mathcal{J}_q$ in $H^1(\R^3)$ is a critical point of $\mathcal{J}_q$.
\end{Lem}
\begin{proof}
	Let $\{v_n\}\subset H^1(\R^3)$ be a (PS) sequence  for $\mathcal{J}_q$ and $v$ its weak limit. Then, for all $\varphi\in C_0^\infty(\R^3)$ we have that
	\[
	\mathcal{J}_q'(v_n)[\varphi] = \int \nabla v_{n} \nabla \varphi+\omega \int v_{n}\varphi
	+q^{2}\int \phi_{v_{n}}v_{n}\varphi - \int|v_{n}|^{p-2}v_{n}\varphi
	\to 0
	\quad
	\hbox{ as } n\to +\infty.
	\]
	Due to the 
	strong convergence of $v_n$ to $v$ in $L_{\rm loc}^q(\R^3)$ for $1\leq q < 6$, to conclude it is enough to prove that
	\begin{equation*}
	\int \phi_{v_{n}}v_n\varphi
	\to
	\int \phi_{v}v\varphi.
	\end{equation*}
	Observe that
	\[
	\left|\int \phi_{v_{n}}v_n\varphi
	- \int \phi_{v}v\varphi\right|
	\leq
	\underbrace{\int \phi_{v_{n}}|v_n - v| |\varphi |}_{I_1}
	+
	\underbrace{\int |\phi_{v_{n}} - \phi_{v} | |v\varphi |}_{I_2}.
	\]
	By the H\"older inequality, the boundedness of $\{\phi_{v_n}\}$ in $L^6(\R^3)$, see (\ref{propphivii}) in Lemma \ref{lem:propphi}, 
	and the strong convergence of $v_n$ to $v$ in $L_{\rm loc}^3(\R^3)$  we get
	\[
	I_1 \leq \| \phi_{v_{n}}\|_6 \|v_n - v\|_{L^2(\operatorname{supp}\varphi)} \|\varphi \|_{L^3(\operatorname{supp}\varphi)}
	\to 0
	\quad
	\hbox{ as } n\to +\infty.
	\]
	On the other hand, using  Lemma \ref{lem:propphi}, item \eqref{propphii} we infer
	\[
	I_2 \leq \|\phi_{v_{n}} - \phi_{v} \|_{L^2(\operatorname{supp}\varphi)} \|v\varphi \|_{L^2(\operatorname{supp}\varphi)}\to 0
	\quad
	\hbox{ as } n\to +\infty,
	\]
	completing the proof by density.
\end{proof}

\begin{Lem}\label{lem:BLforJ}
	For every $v\in H^1(\R^3)$ and $v_n \rightharpoonup0 $ in $H^1(\R^3)$, we have
	\[
	\mathcal{J}_q (v_n + v) - \mathcal{J}_q (v_n) - \mathcal{J}_q (v) \to 0
	\quad
	\hbox{ as } n\to +\infty.
	\]
\end{Lem}
\begin{proof}
	By $\|v_n + v\|^2 = \|v_n\|^2 + \|v\|^2 + o_n(1)$ and  the Brezis-Lieb Lemma it is enough  to show that
	\[
	\int \phi _{v_{n}+v}(v_n + v)^2
	-\int \phi_{v_{n}}v_n^2
	-\int \phi_{v}v^2
	\to 0
	\quad
	\hbox{as }n\to +\infty.
	\]
	But
	\begin{align*}
	\int \phi _{v_{n}+v}(v_n + v)^2
	-\int \phi_{v_{n}}v_n^2
	-\int \phi_{v}v^2
	&=
	4\int  \phi_{v_{n}}v_n v
	+ 2\int \phi_{v_{n}}v^2
	+4\int (\mathcal{K}*v_n v)v_n v
	+4\int \phi_{v} v_n v,
	\end{align*}
	and each term in the right hand side above  converges to zero. Let us see the proof of the second one, being
	the proof of the other terms completely analogous.\\	
	For a subset $A\subset \mathbb R^{3}$ let us denote with $\mathbf 1_A$ its characteristic function.
	Let $B_1$ and $B_2$ be two spheres centered in $0$ with radius $R_1$ and $R_2$.
	We first write 
	\[
	\int  \phi_{v_{n}}v^2
	=
	\int (\mathcal{K}* \mathbf 1_{B_{1}} v_n^2) v^2
	+\int (\mathcal{K}*\mathbf 1_{B_{1}^{c}}v_n^2)\mathbf 1_{B_{2}} v^2
	+\int (\mathcal{K}*\mathbf 1_{B_{1}^{c}} v_n^2)\mathbf 1_{B_{2}^{c}}v^2.
	\]
	Then, since $\mathcal{K}\leq 1/a$, 
	we easily get
	\[
	\int (\mathcal{K}* \mathbf 1_{B_{1}} v_n^2) v^2
	\leq
	\frac{1}{a}\|v_n\|_{L^{2}(B_1)}^2 \|v\|_{2}^2
	\to 0,
	\]
	\[
	\int (\mathcal{K}*\mathbf 1_{B_{1}^{c}}v_n^2)\mathbf 1_{B_{2}^{c}}v^2
	\leq
	\frac{1}{a}\|v_n\|_{2}^2 \|v\|_{L^{2}(B_2^c)}^2
	<\frac{1}{n},
	\]
	if $R_2=R_2(n)$ is taken sufficiently large, and,  using that $\mathcal{K}\leq |x|^{-1}$,
	\[
	\int (\mathcal{K}*\mathbf 1_{B_{1}^{c}}v_n^2)\mathbf 1_{B_{2}} v^2
	\leq
	\iint_{B_2 \times B_1^c} \frac{v_n^2(y) v^2(x)}{|x-y|}\,dxdy
	\leq \frac{\|v\|_2^2 \|v_n\|_2^2}{|R_1 - R_2|}
	<\frac{1}{n}
	\]
	taking $R_1=R_1(n)$ sufficiently large. 
\end{proof}

Let us recall the Lions Lemma
\begin{Lem}\label{lem:Lions}
	Let $2<r<6$. There exists a constant $C>0$ such that
	$$\forall u\in H^{1}(\mathbb R^{3}): \|u\|_{r}\leq \left( \sup_{z\in\mathbb Z^{3}}\|u\|_{L^{2}(z+Q)}\right)^{(r-2)/r} \|u\|^{2/r}$$
	where $Q=[0,1]^{3}$.
\end{Lem}

As a consequence of this lemma we infer
\begin{Lem}\label{lem:generale}
	Let $\{v_{n}\}\subset H^{1}(\mathbb R^{3})$ be a sequence such that  $v_{n}\rightharpoonup 0$ in $H^{1}(\mathbb R^{3})$. Then $\mathcal J_{q}'(v_{n})\to 0$.\\
	If, in addition, $v_{n}\not\to0$ in $H^{1}(\mathbb R^{3})$,
	then, up to subsequences,
	\begin{equation*}
	\exists \{z_{n}\}\subset\mathbb Z^{3} \text{ with } |z_{n}|\to +\infty  \hbox{ such that } \lim_{n}\|v_{n}\|_{L^{p}(z_{n}+Q)}>0.
	\end{equation*}
\end{Lem}
\begin{proof}
	Let $\varphi\in C^{\infty}_{c}(\mathbb R^{3})$ and $\Omega:=\textrm{supp } \varphi$.
	We have
	$$\mathcal J_{q}'(v_{n})[\varphi] = \langle v_{n}, \varphi \rangle + \int_{\Omega} \phi_{v_{n}} v_{n}\varphi
	-\int_{\Omega} |v_{n}|^{p-2}v_{n}\varphi.$$
	Then the first part  follows by observing that
	\begin{align*}
	&\left |  \int_{\Omega} \phi_{v_{n}}  v_{n}\varphi  \right|
	\leq  \frac{1}{a}\|v_{n} \|_{2}^{2} \|v_{n}\|_{L^{2}(\Omega)} \|\varphi \|_{L^{2}(\Omega)} = o_{n}(1), \\
	&\left | \int_{\Omega} |v_{n}|^{p-2} v_{n} \varphi  \right| \leq \| v_{n}\|^{p-1}_{L^{p'}(\Omega)} 
	\|\varphi \|_{L^{p}(\Omega)} =o_{n}(1),
	\end{align*}
	uniformly in $\varphi$.
	We conclude by density.\\
	Assume now that $v_{n}\not\to0$. Then there exists a subsequence, that we rename again $v_{n}$,
	such that $\|v_{n}\|\to \alpha>0.$
	If $\liminf_{n} \sup_{z\in \mathbb  Z^{3}} \|v_{n}\|_{L^{p}(z+Q)}=0$,
	the Lions Lemma \ref{lem:Lions} gives $\liminf_{n} \|v_{n}\|_{p} =0$ and then
	we have
	$$0<\alpha= \liminf_{n}\|v_{n}\|^{2}\leq  \lim_{n} \mathcal J_{q}'(v_{n})[v_{n}]+
	\liminf_{n} \|v_{n}\|_{p}^{p} = 0,$$
	reaching a contradiction.
	Hence
	$\liminf_{n} \sup_{z\in \mathbb  Z^{3}} \|v_{n}\|_{L^{p}(z+Q)}>0$ and then
	there exists a sequence $\{z_{n}\}\subset \mathbb Z^{3}$ such that
	\begin{equation*}
	\lim_{n}\| v_{n}\|_{L^{p}(z_{n} +Q)} >0.
	\end{equation*}
	The sequence $\{z_{n}\}$ has to be unbounded.
	Otherwise,
	if for some $R>0$ it is
	$z_{n}+Q\subset B_{R}$
	for all $n\in \mathbb N$, we have the contradiction
	$$0<\lim_{n}\|v_{n}\|_{L^{p}(z_{n}+Q)}\leq \lim_{n}\|v_{n}\|_{L^{p}(B_{R})}=0,$$
	concluding the proof.
\end{proof}
Finally we recall two basic facts.
\begin{Lem} \label{lem:ovvio}
	Let $\{y_{n}\}\subset \mathbb R^{3}$, $v\in H^{1}(\mathbb R^{3})$, $\{v_{n}\}\subset H^{1}(\mathbb R^{3})$ be bounded.
	\begin{enumerate}[label=(\roman*),ref=\roman*]
		\item \label{ovvio1} If $|y_{n}|\to +\infty$,
		then $v(\cdot+y_{n})\rightharpoonup 0$ in $H^{1}(\mathbb R^{3})$. 
		\item \label{ovvio2} If $\{y_{n}\}$ is bounded, then, up to a subsequence,
		$$v_{n}\not\rightharpoonup 0 \ \text{ in } \ H^{1}(\mathbb R^{3}) \ \Longrightarrow \ 
		v_{n}(\cdot + y_{n})\not\rightharpoonup 0 \ \text{ in } \ H^{1}(\mathbb R^{3}).$$
	\end{enumerate}
\end{Lem}
\begin{proof}
	For the first part, if $w\in H^{1}(\mathbb R^{3})$ and $\varepsilon>0$, then there exists $\varphi_{\varepsilon}\in C^{\infty}_{c}(\mathbb R^{3})$ such that $\|w-\varphi_{\varepsilon}\|\leq \varepsilon$. Consequently
	\[
	|\langle v(\cdot+y_{n}), w \rangle|
	\leq
	|\langle v(\cdot+y_{n}), w-\varphi_{\varepsilon}\rangle| +| \langle v(\cdot+y_{n}), \varphi_{\varepsilon}\rangle|
	\leq
	\varepsilon \| v\|  + o_{n}(1)
	\]
	proving that $\limsup_{n} |\langle v(\cdot +y_{n}), w\rangle |\leq \varepsilon\|v\|$.  \\
	To show the second part, let $\varphi\in C^{\infty}_{c}(\mathbb R^{3})$ and $y\in \mathbb R^3$ be such that 
	$\langle v_{n}, \varphi\rangle \to \eta \neq 0$ and $y_{n}\to y$. We have
	\[
		\langle v_{n}(\cdot+y_{n}), \varphi(\cdot+y) \rangle
		= \langle v_{n}, \varphi(\cdot +y-y_{n})\rangle
		=\langle v_{n}, \varphi \rangle + \langle v_{n}, \varphi(\cdot +y-y_{n})-\varphi\rangle.
		\]
		Moreover, by the Lebesgue Theorem,
	\[
	|\langle v_{n}, \varphi(\cdot +y-y_{n})-\varphi\rangle|
	\leq C \|\varphi(\cdot -y_{n}) - \varphi(\cdot-y)\|
	=C \|\varphi(\cdot -y_{n}) - \varphi(\cdot-y)\|_{H^1(K)}
	=o_n(1)
	\]
	where $K\subset\R^3$ is a suitable compact set,	completing the proof.
\end{proof}

%

Now we are able to give the proof of Lemma \ref{lemSplitting}.

\begin{proof}[Proof of Lemma \ref{lemSplitting}]
	By Lemma \ref{lem:wlimit} we know that
	$\mathcal J'_{q}(u_{0}) = 0$.
	Let us divide the proof in various steps.\\
	{\bf STEP	1:} We have two possibilities.\\
	{\bf Case 1a:}
	If $u_{n}\to u_{0}$ in $H^{1}(\mathbb R^{3})$, then the first alternative in the Lemma follows
	and the proof is concluded.\\
	{\bf Case 1b:}
	If $u_{n}\not\to u_{0}$ in $H^{1}(\mathbb R^{3})$, then we set
	$u_{n}^{(1)} := u_{n} - u_{0}$, which  satisfies, in view of 
	Lemma \ref{lem:generale} and  Lemma \ref{lem:BLforJ}, we have:
	\begin{enumerate}[label=(1b\roman*),ref=1b\roman*]	
		\item \label{b1i}$u_{n}^{(1)}\rightharpoonup 0$ in $H^{1}(\mathbb R^{3})$,
		\item \label{b1ii}$\mathcal J'_{q}(u^{(1)}_{n})\to 0$,
		\item \label{b1iii}$\mathcal J_{q}(u_{n}^{(1)})  \to d-\mathcal J_{q}(u_{0})$.
	\end{enumerate}
	Moreover, again by Lemma \ref{lem:generale}, we have that
	\begin{equation}\label{eq:zn1}
	\exists \{z_{n}^{(1)}\}\subset \mathbb Z^{3}\text{ with } |z_{n}^{(1)}|\to +\infty
	\hbox{ such that }
	\lim_{n}\|u_{n}^{(1)}\|_{L^{p}(z_{n}^{(1)}+Q)}>0.
	\end{equation}
	Setting $\widetilde u_{n}^{(1)}:= u_{n}^{(1)} (\cdot -z_{n}^{(1)})$, we easily get from (\ref{b1i})--(\ref{b1iii}) and \eqref{eq:zn1} that
	$$ \{\widetilde u_{n}^{(1)}\} \ \text{ is bounded in $H^{1}(\mathbb R^{3})$}, \quad 
	\mathcal J_{q}'(\widetilde u_{n}^{(1)})\to 0, \qquad \widetilde u_{n}^{(1)}\not\rightharpoonup 0 \text{ in }H^{1}(\mathbb R^{3}).$$
	Then 
	\begin{equation*}
	\widetilde u_{n}^{(1)} \rightharpoonup w_{1}\neq 0 \text{ in } H^{1}(\mathbb R^{3})
	\end{equation*}
	and, by the invariance under translations of the functional and (\ref{b1iii}) we have
	\begin{equation}\label{eq:convc}
	\mathcal J_{q}(\widetilde u_{n}^{(1)})=\mathcal J_{q}( u_{n}^{(1)})\to d - \mathcal J_{q}(u_{0}),
	\end{equation}
	so that $\{\widetilde u_{n}^{(1)}\}$ is a bounded (PS) sequence for $\mathcal J_{q}$.
	By Lemma \ref{lem:wlimit}, 
	\begin{equation*}
	\mathcal J_{q}'( w_{1})=0 \quad \text{ with } w_{1}\neq0.
	\end{equation*}
	{\bf STEP 2:} Now there are two possibilities.\\
	{\bf Case 2a:} If $\widetilde u_{n}^{(1)}\to w_{1}$ in $H^{1}(\mathbb R^{3})$, this means that 
	\begin{equation*}
	o_{n}(1) = \|u_{n}^{(1)} - w_{1}(\cdot + z_{n}^{(1)})\| = \| u_{n} - u_{0} - w_{1}(\cdot + z_{n}^{(1)})\|
	\end{equation*}
	and  then $\mathcal J_{q}(\widetilde u_{n}^{(1)}) =\mathcal J_{q}( u_{n}^{(1)}) \to \mathcal J_{q}(w_{1})$,
	which, taking into account \eqref{eq:convc} gives
	\begin{equation*}
	d= \mathcal J_{q}(u_{0}) +\mathcal J_{q}(w_{1})
	\end{equation*}
	and the Lemma is proved with $\ell=1$.\\
	{\bf Case 2b:}  If $\widetilde u_{n}^{(1)}\not\to w_{1}$, then let
	$u_{n}^{(2)}:=  u_{n}^{(1)} - w_{1}(\cdot + z_{n}^{(1)})\not\to0$.
	The sequence $\{u_{n}^{(2)}\}$ satisfies:
	\begin{enumerate}[label=(2b\roman*),ref=2b\roman*]	
		\item \label{2bi} $u_{n}^{(2)}\rightharpoonup 0$ in $H^{1}(\mathbb R^{3})$,
		\item \label{2bii} $\mathcal J'_{q}(u^{(2)}_{n})\to 0$, 
		\item \label{2biii} $\mathcal J_{q}(u_{n}^{(2)}) \to d-\mathcal J_{q}(u_{0})-\mathcal J_{q}(w_{1})$, since, by Lemma \ref{lem:BLforJ},
		\begin{equation*}
		\mathcal J_{q}(u_{n}^{(2)})
		=
		\mathcal J_{q}(\widetilde u_{n}^{(1)}-w_{1}) 
		=
		\mathcal J_{q}(\widetilde u_{n}^{(1)} ) -\mathcal J_{q}(w_{1})+o_{n}(1) 
		=
		d-\mathcal J_{q}(u_{0})-\mathcal J_{q}(w_{1})+o_{n}(1).
		\end{equation*}
	\end{enumerate}
	Again we have also that
	\begin{equation}\label{eq:zn2}
	\exists \{z_{n}^{(2)}\}\subset \mathbb Z^{3}\text{ with } |z_{n}^{(2)}|\to +\infty
	\hbox{ such that }
	\lim_{n}\|u_{n}^{(2)}\|_{L^{p}(z_{n}^{(2)}+Q)}>0.
	\end{equation}
	Setting $\widetilde u_{n}^{(2)} := u_{n}^{(2)}(\cdot -z_{n}^{(2)})$, it holds as before that
	$$ \{\widetilde u_{n}^{(2)}\} \ \text{ is bounded in $H^{1}(\mathbb R^{3})$}, \quad 
	\mathcal J_{q}'(\widetilde u_{n}^{(2)})\to 0, \qquad \widetilde u_{n}^{(2)}\rightharpoonup w_{2}\neq0 \text{ in }H^{1}(\mathbb R^{3}).$$
	Then $\{\widetilde u_{n}^{(2)}\}$ is a bounded (PS) sequence for $\mathcal J_{q}$ and
	by Lemma \ref{lem:wlimit}, 
	\begin{equation}\label{eq:w2}
	\mathcal J_{q}'( w_{2})=0 \quad \text{ with } w_{2}\neq0.
	\end{equation}
	Moreover
	\begin{equation}\label{eq:12divergente}
	|z_{n}^{(1)} - z_{n}^{(2)}|\to +\infty.
	\end{equation}
	To see this, first observe that
	\begin{equation*}
	\widetilde u_{n}^{(1)}-w_{1} = u_{n}^{(1)}(\cdot -z_{n}^{(1)})  - w_{1}= u_{n}^{(2)}(\cdot-z_{n}^{(1)})=
	\widetilde u_{n}^{(2)}(\cdot +z_{n}^{(2)} - z_{n}^{(1)} ).
	\end{equation*}
	Then if it were $|z_{n}^{(2)} - z_{n}^{(1)}|\leq R$, 
	since $\widetilde u_{n}^{(2)}\not\rightharpoonup 0$, by Lemma \ref{lem:ovvio} item \eqref{ovvio2}, we deduce
	$$\widetilde u_{n}^{(2)}(\cdot +z_{n}^{(2)} - z_{n}^{(1)} )\not\rightharpoonup 0,$$ which is a contradiction.\\
	%
	{\bf STEP3:} Again we have two possibilities.\\
	{\bf Case 3a:} If $\widetilde u_{n}^{(2)}\to w_{2}$ in $H^{1}(\mathbb R^{3})$ 
	this means that
	\begin{equation}\label{eq:convergenza2}
	\begin{split}
	o_{n}(1)
	&=
	\| \widetilde u_{n}^{(2)} - w_{2}\| \\
	&=
	\| u_{n}^{(2)}-w_{2}(\cdot +z_{n}^{(2)})\| \\ 
	&=
	\| u_{n}^{(1)}-w_{1}(\cdot +z_{n}^{(1)})-w_{2}(\cdot +z_{n}^{(2)})\|  \\ 
	&=
	\| u_{n} -u_{0}-w_{1}(\cdot +z_{n}^{(1)})-w_{2}(\cdot +z_{n}^{(2)})\| 
	\end{split}
	\end{equation}
	and  then, being $\mathcal J_{q}(\widetilde u_{n}^{(2)}) =\mathcal J_{q}( u_{n}^{(2)}) $
	the Lemma holds, in virtue of 
	\eqref{eq:zn2}--\eqref{eq:convergenza2}
	with $\ell=2$.\\
	{\bf Case 3b:} If $\widetilde u_{n}^{(2)}\not\to w_{2}$, we argue as before repeating the procedure.\\
	In this way we obtain at the generic\\
	{\bf STEPm} with the following alternatives:\\
{\bf Case ma:} $\widetilde u_{n}^{(m-1)}\to w_{m-1}$ in $H^{1}(\mathbb R^{3})$ and the Lemma holds with $\ell=m-1$.\\
{\bf Case mb:} We have
	\begin{itemize}
		\item sequences of points $\{z_{n}^{(i)}\}\subset \mathbb R^{3}$ for $i=1,\ldots, m$ with $|z_{n}^{(i)}|\to +\infty$ for all $i=1,  \ldots,m$ and $|z_{n}^{(i)} - z_{n}^{(j)}|\to +\infty$ for all $i,j=1,\ldots,m$ with $i\neq j $;
		\item  functions $w_{i}\neq 0$ with $\mathcal J'_{q}(w_{i}) = 0$ for all $i=1,\ldots, m$;
	\end{itemize}
	and in this case the procedure continues.\\
%
However at some step ($\ell+1$) the first case has to occur stopping the process
	and proving the Lemma. That is
	there exists $\ell\in \mathbb N$
	such that $\widetilde u_{n}^{(\ell)}\to w_{\ell}$.
	To see this, we first observe that, for any $N\in \mathbb N$ we have that
	\begin{equation}\label{eq:normanorma}
	\Big\| u_{n} - u_{0} -\sum_{i=1}^{N}
	w_{i}(\cdot + z_{n}^{(i)})\Big\|^{2}=\|u_{n}\|^{2} - \|u_{0}\|^{2} -\sum_{i=1}^{N}\| w_{i}\|^{2} +o_{n}(1)
	\end{equation} 
	Indeed expanding the left hand side above we have
	\begin{equation}\label{eq:bo}
	\begin{split}
	\Big\| u_{n} - u_{0} -\sum_{i=1}^{N} w_{i}(\cdot + z_{n}^{(i)})\Big\|^{2}
	&=
	\|u_{n}\|^{2}+\|u_{0}\|^{2} +
	\Big\|\sum_{i=1}^{N}w_{i}(\cdot + z_{n}^{(i)})\Big\| ^{2}
	-2\langle u_{n}, u_{0}\rangle  \\
	&\quad
	-2\sum_{i=1}^{N}\langle u_{n}, w_{i}(\cdot + z_{n}^{(i)})\rangle 
	- 2\sum_{i=1}^{N}\langle u_{0},  w_{i}(\cdot + z_{n}^{(i)})\rangle.
	\end{split}
	\end{equation}
	Now, since $|z_{n}^{(i)} - z_{n}^{(j)}| \to +\infty$ for $i\neq j$, by (\ref{ovvio1}) in Lemma \ref{lem:ovvio},
	\begin{equation}\label{eq:contofinale1}
	\Big\|\sum_{i=1}^{N} w_{i}(\cdot + z_{n}^{(i)})\Big\|^{2}=
	\sum_{i=1}^{N} \Big\|w_{i}(\cdot + z_{n}^{(i)})\Big\| ^{2} +
	2 \sum_{i\neq j} \langle w_{i}(\cdot +z_{n}^{(i)}), w_{j}(\cdot +z_{n}^{(j)}) \rangle
	= \sum_{i=1}^{N} \|w_{i}\| ^{2} + o_{n}(1).
	\end{equation}
	Analogously, being $|z_{n}^{(i)}|\to+\infty$, 
	\begin{eqnarray}\label{eq:contofinale2}
	\langle u_{0}, w_{i}(\cdot + z_{n}^{(i)})\rangle =o_{n}(1).
	\end{eqnarray}
	Finally,
	\begin{equation}\label{eq:contofinale3}
	\begin{split}
	\langle u_{n}, w_{i}(\cdot+z_{n}^{(i)}) \rangle
	=&
	\langle u_{n} - u_{0} -\sum_{j=1}^{i-1}w_{j}(\cdot+z_{n}^{(j)}),
	w_{i}(\cdot +z_{n}^{(i)})\rangle 
	+\langle  u_{0} +\sum_{j=1}^{i-1}w_{j}(\cdot+z_{n}^{(j)}), w_{i}(\cdot +z_{n}^{(i)})\rangle \\
	&=
	\langle u_{n}^{(i)} , w_{i}(\cdot +z_{n}^{(i)}) \rangle 
	+\langle u_{0}, w_{i}(\cdot +z_{n}^{(i)})\rangle
	+\sum_{j=1}^{i-1}\langle w_{j}, w_{i}(\cdot+ z_{n}^{(i)} - z_{n}^{(j)})\rangle \\
	&=
	\langle \widetilde u_{n}^{(i)}, w_{i}\rangle +o_{n}(1) \\
	&=
	\|w_{i}\|^{2}+o_{n}(1).
	\end{split}	
	\end{equation}
	Then by plugging \eqref{eq:contofinale1}-\eqref{eq:contofinale3} into \eqref{eq:bo} we get
	\eqref{eq:normanorma}.\\
	In virtue of the fact that $w_i$ are nontrivial critical points of $\J_q$, we have
	\[
	\|w_i\|^2 \leq \|w_i\|^2 + \int \phi_{w_{i}} w_i^2 = \| w_i\|_p^p \leq C \|w_i\|^p
	\]
	showing that $\{w_{i}\}$ are bounded away from zero in $H^{1}(\mathbb R^{3})$.\\
	Then, by \eqref{eq:normanorma} we deduce that the process has to stop, completing the proof of the Lemma.
\end{proof}


\end{document}